\documentclass[11pt]{article}

\usepackage{amsmath,amssymb,amsthm,mathtools}
\usepackage{fullpage}
\usepackage{enumitem}
\usepackage{microtype}
\usepackage{cite} 
\usepackage{hyperref}

\newcommand{\orcid}[1]{\href{https://orcid.org/#1}{\texttt{#1}}}

\makeatletter
\renewcommand\@biblabel[1]{#1.}
\makeatother

\newcommand{\R}{\mathbb{R}}
\newcommand{\E}{\mathbb{E}}
\newcommand{\Pbb}{\mathbb{P}}
\newcommand{\Hb}{\mathbb{H}}
\newcommand{\Vb}{\mathbb{V}}
\newcommand{\1}{\mathbf{1}}
\newcommand{\law}{\mathcal{L}}
\newcommand{\dd}{\mathrm{d}}
\newcommand{\dx}{\,\dd x}
\newcommand{\dt}{\,\dd t}

\newcommand{\supp}{\operatorname{supp}}

\newcommand{\Xstar}{X^{\ast}}
\newcommand{\Xstarn}{X^{(\ast,n)}}

\theoremstyle{plain}
\newtheorem{theorem}{Theorem}[section]
\newtheorem{proposition}[theorem]{Proposition}
\newtheorem{lemma}[theorem]{Lemma}
\newtheorem{corollary}[theorem]{Corollary}

\theoremstyle{definition}
\newtheorem{assumption}[theorem]{Assumption}

\theoremstyle{remark}
\newtheorem{remark}[theorem]{Remark}

\title{Existence of densities and atoms for the running maximum of time-inhomogeneous jump diffusions}
\author{Takuya Nakagawa\thanks{Department of Mathematical Sciences,
Ritsumeikan University, 1-1-1 Noji-Higashi, Kusatsu, Shiga 525-8577, Japan.
Email: \texttt{takuya.nakagawa73@gmail.com}.
ORCID: \orcid{0000-0003-1216-4822}.}
\and
Ryoichi Suzuki\thanks{Department of Business Economics, School of Management,
Tokyo University of Science, 1-11-2 Fujimi, Chiyoda-ku, Tokyo 102-0071, Japan.
Email: \texttt{rsuzukimath@gmail.com}.
ORCID: \orcid{0000-0001-9979-1882}.
Corresponding author.}
}
\date{}

\begin{document}
\maketitle

\begin{abstract}
We prove absolute continuity of the running maximum $\Xstar_T=\sup_{0\le s\le T}X_s$ of one-dimensional time-inhomogeneous L\'evy--It\^o diffusions driven by a Brownian motion and an independent non-truncated pure-jump L\'evy process.
Using Bismut's directional Malliavin calculus on the Wiener--Poisson space together with the running-maximum criteria of Song--Xie and Nakagawa--Suzuki, we reduce the problem to constructing an admissible direction $\Theta$.
The key is to ensure that the directional derivative $D_\Theta X_t$ is strictly positive for all $t\in(0,T]$.
We give explicit directions in two regimes.
In the uniformly elliptic case with time-inhomogeneous coefficients, a purely Brownian perturbation yields an explicit positive integral representation for $D_\Theta X_t$, and hence $\Xstar_T$ admits a density without truncating the jump component.
In a Brownian-degenerate pure-jump model with a bounded deterministic time-dependent jump weight $\kappa(t)$ that may vanish on subintervals, we prove absolute continuity under the minimal nondegeneracy-in-time condition $\int_0^t \kappa(s)^2\,ds>0$ for every $t>0$ and infinite activity of the L\'evy measure.
A weighted Poisson positivity lemma is the key new input.
Finally, we show that silent initial intervals can create atoms and we derive an explicit atom--density decomposition.
\end{abstract}


\section{Introduction}

Let $X=(X_t)_{t\in[0,T]}$ be a one-dimensional jump diffusion and denote its running maximum by
\[
\Xstar_T := \sup_{0\le s\le T} X_s.
\]
Running maxima are central objects in fluctuation theory and appear in applications through extrema-based risk measures and
path-dependent payoffs (barrier and lookback functionals); see, e.g., \cite{Bertoin1996,ContTankov2004,Kyprianou2014}.
For L\'evy processes, the law of the supremum can often be analyzed via Wiener--Hopf factorization.
Beyond explicit factorization, the absolute continuity of the supremum (and related joint laws) has been studied in depth; see Chaumont \cite{Chaumont2013} and, for PDE characterizations and density representations, Coutin, Pontier and Ngom \cite{CoutinPontierNgom2018}.
For solutions of SDEs driven by L\'evy noise, such explicit fluctuation identities are typically unavailable,
and it becomes natural to study qualitative regularity properties of $\law(\Xstar_T)$, in particular the existence of a density.

We focus on time-inhomogeneous L\'evy--It\^o diffusions solving
\begin{equation}\label{eq:SDE-general}
\dd X_t=b(t,X_t)\dt+\sigma(t,X_t)\,\dd W_t+\kappa(t)\,\dd L_t,\qquad X_0=x\in\R.
\end{equation}
Here $W$ is a Brownian motion, and $L$ is an independent, non-truncated, pure-jump L\'evy process with L\'evy measure $\nu$.
The jump amplitude $\kappa:[0,T]\to\R$ is a bounded deterministic measurable function; in the Brownian-degenerate regime it may vanish on subintervals.
Precise assumptions on the coefficients $b$, $\sigma$, $\kappa$ and on $\nu$ are stated in Sections~\ref{sec:setting} and~\ref{sec:main}. Depending on the source of nondegeneracy, we treat separately a uniformly elliptic Brownian regime and a Brownian-degenerate regime in which the jump term (weighted by the time-dependent coefficient $\kappa$) provides the relevant nondegeneracy.
In the Brownian-degenerate regime, the minimal nondegeneracy-in-time condition reads $\int_0^t\kappa(s)^2\,\dd s>0$ for every $t>0$.

The supremum functional $(\xi_s)_{s\in[0,T]}\mapsto \sup_{s\le T}\xi_s$ is Lipschitz (on the space of bounded paths) but highly nonsmooth.
Therefore density questions for $\Xstar_T$ are not direct consequences of the classical Malliavin calculus for smooth functionals.
A stochastic calculus of variations for jump processes was initiated by Bismut \cite{Bismut1983} and further developed in several directions;
see \cite{Bichteler1987,IshikawaKunita2006,Ishikawa2023,Nualart2006}.
These tools yield smooth densities for $X_T$ under suitable nondegeneracy conditions (e.g.\ \cite{BallyClement2011,Cass2009,Picard1996}),
but the running maximum requires additional ideas because it depends on the whole path and is not a smooth functional.
Related regularity results for densities of SDEs driven by (possibly degenerate) additive L\'evy noises, based on Bismut-type Malliavin calculus, can be found in Song and Zhang \cite{SongZhang2015}.

Even for continuous diffusions, regularity of laws involving the running maximum has been investigated using Malliavin calculus
and integration-by-parts formulas; see, e.g., Hayashi and Kohatsu-Higa \cite{HayashiKohatsuHiga2013},
Nakatsu \cite{Nakatsu2016,Nakatsu2019}, and Coutin and Pontier \cite{CoutinPontier2019,CoutinPontier2023}.

Song and Xie \cite{SongXie2018} provided a Bismut--Malliavin criterion for the existence of a density for the running maximum of Wiener--Poisson functionals
and applied it to L\'evy--It\^o diffusions driven by a Brownian motion and a \emph{truncated} symmetric stable process.
Passing from truncated to non-truncated (infinite-activity) pure-jump L\'evy noise is delicate: truncation produces only weak convergence,
and products of weakly convergent sequences may fail to converge.
Nakagawa and Suzuki \cite{NakagawaSuzuki2024} addressed this difficulty and established absolute continuity of $\Xstar_T$ for the additive-coefficient model~\eqref{eq:NSmodel}
\begin{equation}\label{eq:NSmodel}
\dd X_t=b(X_t)\dt+\sigma_1\,\dd W_t+\sigma_2\,\dd L_t,
\end{equation}
with constant $\sigma_1,\sigma_2$ and a non-truncated pure-jump L\'evy process $L$ under conditions that cover, in particular, symmetric $\alpha$-stable processes with $\alpha\in(1,2)$.
A key input is the running-maximum criterion \cite[Lemma~4.1]{NakagawaSuzuki2024}, based on weak compactness of the directional Sobolev space $W^{1,p}_\Theta$
for Wiener--Poisson functionals.

\medskip
The present paper takes a complementary viewpoint.
We take the running-maximum criterion of \cite{NakagawaSuzuki2024} as an input and focus on constructing, directly for the \emph{original} model,
an admissible direction $\Theta$ such that the directional derivative $D_\Theta X_t$ is strictly positive for all $t\in(0,T]$.
This yields a unified scheme in which we do not differentiate the nonsmooth supremum functional; see Proposition~\ref{prop:unified}.
We provide two explicit nondegeneracy mechanisms.
In a time-inhomogeneous jump diffusion with a uniformly elliptic, state-dependent diffusion coefficient,
a Brownian direction alone makes $D_\Theta X_t$ an explicit positive Lebesgue integral, which in particular avoids any truncation of the jump component.
In a Brownian-degenerate model we allow a time-dependent jump weight that may vanish on subintervals;
under a minimal nondegeneracy-in-time condition, a weighted Poisson positivity lemma forces $D_\Theta X_t>0$ despite the possible vanishing of the weight.
Finally, we show that silent initial intervals may create atoms and we derive an explicit atom--density decomposition,
demonstrating that the nondegeneracy-in-time condition is essentially sharp.

Technically, our approach separates the argument into an \emph{analytic} part (moment bounds and membership in $W^{1,p}_\Theta$)
and a \emph{geometric} part (explicit construction of a strictly positive direction).
This modularity facilitates the treatment of time-inhomogeneous coefficients and accommodates intermittency arising from vanishing noise amplitudes.

\medskip
\noindent\textbf{Relation to previous work.}
Compared with Song and Xie \cite{SongXie2018} and Nakagawa and Suzuki \cite{NakagawaSuzuki2024}, our results add three features:
\begin{itemize}[leftmargin=2.2em]
\item[(R1)] We work throughout with the \emph{non-truncated} pure-jump L\'evy noise.
In the elliptic regime, strict positivity is obtained by a purely Brownian direction, hence avoiding any truncation step as in \cite{SongXie2018}.
\item[(R2)] Beyond the additive constant-coefficient setting of \cite{NakagawaSuzuki2024}, we allow time-inhomogeneous coefficients and state dependence of the elliptic diffusion coefficient, and we permit a time-dependent jump amplitude that may vanish on subintervals.
\item[(R3)] We complement density results by identifying a sharp obstruction:
if both noises are silent on a nontrivial initial interval, then $\law(\Xstar_T)$ may contain an atom at the deterministic initial maximum, and we give an explicit atom--density decomposition.
\end{itemize}

\medskip
\noindent\textbf{Contributions.}
Fix $T>0$.
\begin{itemize}[leftmargin=2.2em]
\item[(C1)] (\emph{Uniformly elliptic Brownian part, time-inhomogeneous coefficients.})
We consider the time-inhomogeneous jump diffusion \eqref{eq:SDE-general} with a uniformly elliptic diffusion coefficient $\sigma(t,x)$.
In this regime we construct an admissible direction $\Theta=(h,0)$ using \emph{only} the Brownian component such that $D_\Theta X_t>0$ for all $t>0$,
and hence $\Xstar_T$ admits a density.
Compared with \cite{NakagawaSuzuki2024}, this allows time dependence and state dependence in the elliptic coefficient and, in particular,
does not require any jump-direction nondegeneracy or truncation of the pure-jump part.
\item[(C2)] (\emph{Pure-jump model with a time-dependent jump weight.})
In the Brownian-degenerate case $\sigma\equiv 0$, we allow the jump weight $\kappa(t)$ to be time-dependent and to vanish on subintervals.
Under the minimal nondegeneracy-in-time condition \eqref{eq:NDtime} and infinite activity of the L\'evy measure,
we construct a jump direction $\Theta=(0,v)$ and prove that $\Xstar_T$ still admits a density.
The new ingredient is a weighted Poisson positivity lemma which ensures $D_\Theta X_t>0$ for all $t>0$ despite possible vanishing of $\kappa$.
\item[(C3)] (\emph{Sharpness via silent intervals.})
We show that if both noises are silent on a nontrivial initial interval, then $\law(\Xstar_T)$ may contain an atom at the deterministic initial maximum.
When the post-silent dynamics satisfies the assumptions of Theorem~\ref{thm:elliptic} or Theorem~\ref{thm:jump}, we obtain an explicit atom--density decomposition on $(m_0,\infty)$,
showing that \eqref{eq:NDtime} (and its elliptic analogue) is essentially sharp.
\end{itemize}

\medskip
\noindent\textbf{Organization.}
Section~\ref{sec:setting} introduces the L\'evy setting and assumptions.
Section~\ref{sec:malliavin} recalls the directional Malliavin calculus and states the running-maximum criterion (Lemma~\ref{lem:criterion}) and the unified sufficient condition (Proposition~\ref{prop:unified}),
which serve as the common engine of the paper.
Section~\ref{sec:main} states the main results. Proofs are given in Sections~\ref{sec:proof-elliptic}--\ref{sec:proof-silent}.
Technical lemmas are collected in Appendix~\ref{sec:appendix}.

\section{L\'evy setting and model}\label{sec:setting}

Let $(\Omega,\mathcal{F},(\mathcal{F}_t)_{t\in[0,T]},\Pbb)$ support a standard Brownian motion $W$
and an independent pure-jump L\'evy process $L$ with L\'evy measure $\nu$.
We write $\R_0:=\R\setminus\{0\}$.
We write $N(\dt,\dd z)$ for the Poisson random measure of $L$ and $\tilde N(\dt,\dd z)=N(\dt,\dd z)-\nu(\dd z)\dt$.
Throughout we take $L$ to have characteristic triplet $(0,0,\nu)$ with respect to the truncation function
$z\1_{\{|z|\le 1\}}$; equivalently, $L$ admits the L\'evy--It\^o decomposition
\begin{equation}\label{eq:LevyIto}
L_t=\int_0^t\int_{|z|\le 1} z\,\tilde N(\dd s,\dd z)+\int_0^t\int_{|z|>1} z\,N(\dd s,\dd z),\qquad t\in[0,T].
\end{equation}
(Any deterministic drift term of $L$ could be absorbed into the drift coefficient $b$.)

\subsection{Assumptions on the L\'evy measure}

We use the class of non-truncated L\'evy measures considered by \cite{NakagawaSuzuki2024}; see Assumption~\ref{ass:Levy} below.

\begin{assumption}[L\'evy measure conditions]\label{ass:Levy}
Assume that $\nu(\dd z)=c(z)\,\dd z$ for a positive function $c\in C^1(\R_0)$ and that there exist constants $\beta>1$ and $C>0$
such that:
\begin{enumerate}[label=(L\arabic*),leftmargin=2.7em]
\item $\displaystyle \int_{\R_0} (1\wedge |z|^2)\,\nu(\dd z)<\infty$;
\item for any $p\in(1,\beta)$, $\displaystyle \lim_{n\to\infty}\int_{|z|>n}|z|^{p}\,\nu(\dd z)=0$;
\item $\displaystyle \sup_{n\in\mathbb{N}}\left|\int_{1\le |z|\le n} z\,\nu(\dd z)\right|<\infty$ and
$\nu(\{0<|z|\le \varepsilon\})=+\infty$ for every $\varepsilon>0$;
\item $\displaystyle \left|\frac{c'(z)}{c(z)}\right|\le C\left(1\vee \frac1{|z|}\right)$ for all $z\ne 0$.
\end{enumerate}
We refer to \cite{Sato2013} for background on L\'evy measures and infinite activity.
Condition (L1) is standard and ensures that the L\'evy--It\^o decomposition~\eqref{eq:LevyIto}
defines a semimartingale; the first part of~(L3) (bounded compensated large-jump integrals) guarantees
that \eqref{eq:LevyIto} converges without an additional drift correction
(see \cite[Theorem~19.2]{Sato2013}).
This is essentially a normalization allowing us to work with the characteristic triplet $(0,0,\nu)$
for the truncation function $z\1_{\{|z|\le 1\}}$; any resulting deterministic drift term can be absorbed into $b$ as noted above.
The infinite-activity part of~(L3) is used
in Lemma~\ref{lem:poisson-positive} ($\nu(\{z:\eta(z)>0\})=\infty$).
Conditions~(L2) and~(L4) are needed for
the Sobolev-space estimates in \cite{NakagawaSuzuki2024}.
\end{assumption}

\begin{remark}[Comparison with the approximation strategy of {\cite{NakagawaSuzuki2024}}]
The proof strategy in \cite{NakagawaSuzuki2024} relies on approximating the non-truncated L\'evy process by truncated processes $L^{(n)}$
and transferring absolute continuity from the finite-jump regime via weak compactness in $W^{1,p}_\Theta$.
In that approximation scheme, constant diffusion and jump coefficients facilitate an explicit computation of $D_\Theta X^{(n)}$
and the uniform estimates needed to pass to the limit.
In contrast, in our uniformly elliptic regime we bypass the truncation argument entirely by working in a Brownian direction and taking the
constant sequence $X^{(n)}\equiv X$; in the pure-jump regime, our main new ingredient is a weighted-Poisson positivity lemma which accommodates
time-dependent weights $\kappa(t)$ that may vanish on subintervals.
\end{remark}

\begin{remark}[Stable example]
If $L$ is a symmetric $\alpha$-stable process with $\alpha\in(1,2)$, then
$c(z)=c_\alpha |z|^{-1-\alpha}$ and Assumption~\ref{ass:Levy} holds with $\beta=\alpha$.
In particular, the first part of~(L3) holds trivially by symmetry, since $\int_{1\le|z|\le n}z\,\nu(\dd z)=0$ for all $n$.
\end{remark}

\subsection{The SDE and well-posedness}

We consider the one-dimensional SDE~\eqref{eq:SDE-general}, that is,
\begin{equation*}
\dd X_t=b(t,X_t)\dt+\sigma(t,X_t)\,\dd W_t+\kappa(t)\,\dd L_t,\qquad X_0=x\in\R.
\end{equation*}
Since $X$ is c\`adl\`ag, one may equivalently evaluate the drift and diffusion coefficients at the left limit $X_{t-}$ in \eqref{eq:SDE-general}.
This leaves the $dt$- and $dW$-integrals unchanged, since the set of jump times is Lebesgue-null; moreover, the jump term is additive with deterministic $\kappa$.
We keep the lighter notation $b(t,X_t)$ and $\sigma(t,X_t)$ throughout.
For $t\in[0,T]$, define $\Xstar_t:=\sup_{0\le s\le t}X_s$.
Precise conditions on the coefficients $b$, $\sigma$ and $\kappa$ are stated separately
for each regime in Section~\ref{sec:main}:
Assumption~\ref{ass:elliptic} for the uniformly elliptic Brownian regime and
Assumption~\ref{ass:jump} for the pure-jump regime.

Under the coefficient assumptions imposed in Section~\ref{sec:main}, the SDE~\eqref{eq:SDE-general} admits a unique strong solution; see, for instance, \cite{Applebaum2009}. Throughout the paper we work with this solution and do not discuss existence and uniqueness further.

\section{Directional Malliavin calculus and a running-maximum criterion}\label{sec:malliavin}

We adopt the directional Malliavin calculus on the Wiener--Poisson space in the spirit of Bismut
and use the framework and notation of \cite{NakagawaSuzuki2024};
see also \cite{BouleauDenis2015} for a comprehensive account of Dirichlet form methods for Poisson measures.
For background on (non-directional) Malliavin calculus on the Wiener--Poisson space and integration-by-parts formulas with jumps, see \cite{BallyClement2011,IshikawaKunita2006,Ishikawa2023}.
In particular, for $p>1$ and $\Theta=(h,v)$ belonging to the admissible spaces
$\Hb_p\times \Vb_p$, the directional derivative $D_\Theta$ and the Sobolev space $W^{1,p}_\Theta$
are defined as in \cite[Definition~3.1]{NakagawaSuzuki2024}.
Informally, $W^{1,p}_\Theta$ consists of random variables (or processes) that possess a directional
derivative $D_\Theta F$ in $L^p$ along the perturbation $\Theta=(h,v)$,
where $h$ perturbs the Brownian path and $v$ perturbs the Poisson jump sizes.

\medskip
\noindent\textbf{Minimal definition (directional derivative).}
Let $\mathcal{S}$ denote the class of smooth Wiener--Poisson cylinder functionals used in \cite[Definition~3.1]{NakagawaSuzuki2024}.
For $F\in\mathcal{S}$, the directional derivative $D_\Theta F$ is defined as the derivative at $\varepsilon=0$
along the Bismut-type perturbation associated with $\Theta=(h,v)$: the Brownian path is shifted in the Cameron--Martin direction $h$
and the jump sizes are transported according to the predictable field $v$.
The Sobolev space $W^{1,p}_\Theta$ is the completion of $\mathcal{S}$ under the norm $\|F\|_{L^p}+\|D_\Theta F\|_{L^p}$.
In the one-dimensional jump direction used later, an explicit expression of the $\Vb_p$-norm and of the weight $\varrho$
appearing in the definition of $\Vb_p$ is recalled in Lemma~\ref{lem:v-admissible}.

For the reader's convenience, we recall that \cite{NakagawaSuzuki2024} allows \emph{random} adapted directions.
For $p>1$, let $\Hb_p$ be the space of all measurable $(\mathcal{F}_t)$-adapted processes $h:\Omega\times[0,T]\to\R$ such that
\[
\|h\|_{\Hb_p}:=\Big(\E\Big[\Big(\int_0^T |h(s)|^2\,\dd s\Big)^{p/2}\Big]\Big)^{1/p}<\infty,
\]
and set $\Hb_{\infty-}:=\bigcap_{p\ge1}\Hb_p$.
Likewise, $\Vb_p$ denotes the space of predictable fields $v:\Omega\times[0,T]\times\R_0\to\R$
equipped with the norm $\|\cdot\|_{\Vb_p}$ in \cite[Section~3]{NakagawaSuzuki2024}, and $\Vb_{\infty-}:=\bigcap_{p\ge1}\Vb_p$.
In particular, directions $(h,v)\in \Hb_p\times \Vb_p$ may depend on $\omega$ (adapted/predictable), whereas the core class of cylinder
functionals used to define $D_\Theta$ is generated by deterministic test functions.

We will rely on the following criterion for absolute continuity of the running maximum.

\begin{lemma}[Criterion for $\Xstar_T$ {\cite[Lemma~4.1]{NakagawaSuzuki2024}}]\label{lem:criterion}
Fix $p>1$ and $\Theta=(h,v)\in \Hb_{\infty-}\times \Vb_{\infty-}$.
For each $n\ge1$, let $\{X^{(n)}_t\}_{t\in[0,T]}$ be a right-continuous real-valued process, and let $\{X_t\}_{t\in[0,T]}$ be a right-continuous real-valued process.
Write
\[
\Xstarn_T:=\sup_{0\le s\le T}X^{(n)}_s,
\qquad
\Xstar_T:=\sup_{0\le s\le T}X_s.
\]

Assume that:
\begin{enumerate}[label=\textup{(\roman*)},leftmargin=2.4em]
\item $\sup_{n\in\mathbb{N}}\E[|\Xstarn_T|^p]<\infty$, and for every $t\in[0,T]$ we have $X^{(n)}_t\in W^{1,p}_\Theta$ with
\[
\sup_{n\in\mathbb{N}}\E\Big[\sup_{0\le s\le T}|D_\Theta X^{(n)}_s|^p\Big]<\infty;
\]
\item for each $n\ge1$, the process $\{D_\Theta X^{(n)}_t\}_{t\in[0,T]}$ admits a right-continuous version;
\item $\displaystyle \lim_{n\to\infty}\E\Big[\sup_{0\le s\le T}|X^{(n)}_s-X_s|^p\Big]=0$.
\end{enumerate}
Assume moreover that $D_\Theta X$ admits a right-continuous version and satisfies
\[
\Pbb\big(D_\Theta X_t\neq 0 \text{ on }\{t\in(0,T]: X_t=\Xstar_T\}\big)=1.
\]
Then the law of $\Xstar_T$ is absolutely continuous with respect to the Lebesgue measure on $\R$.
\end{lemma}

\medskip
Lemma~\ref{lem:criterion} is our main tool to transfer absolute continuity of the running maximum from suitable approximations.
In the sequel, the processes we consider admit c\`adl\`ag versions; in particular they are right-continuous, so the right-continuity requirements in Lemma~\ref{lem:criterion} are met once such versions are fixed.
In applications, it is convenient to enforce a stronger nondegeneracy condition: if one can construct an admissible direction $\Theta$
such that $D_\Theta X_t$ is strictly positive for all $t>0$, then the condition on the set of maximizers required in Lemma~\ref{lem:criterion}
is automatically satisfied.
The next proposition records this simplified sufficient condition, which will be verified in our two main regimes.

\begin{remark}
In Section~\ref{sec:proof-elliptic}, we apply Lemma~\ref{lem:criterion} with the constant sequence $X^{(n)}\equiv X$
and with $\Theta=(h,0)$, i.e., we exploit only the Brownian direction.
\end{remark}

\begin{proposition}[A unified sufficient condition]\label{prop:unified}
Fix $p>1$ and let $\Theta=(h,v)\in \Hb_{\infty-}\times \Vb_{\infty-}$.
Assume that for every $t\in[0,T]$ we have $X_t\in W^{1,p}_\Theta$, that $\Xstar_T\in L^p(\Omega)$, that $D_\Theta X$ admits a right-continuous version with
\[
\E\Big[\sup_{0\le t\le T}|D_\Theta X_t|^p\Big]<\infty,
\]
and that
\[
\Pbb\big(D_\Theta X_t>0,\ \forall t\in(0,T]\big)=1.
\]
Then $\law(\Xstar_T)$ is absolutely continuous with respect to the Lebesgue measure on $\R$.
\end{proposition}

\begin{proof}
Apply Lemma~\ref{lem:criterion} with the constant sequence $X^{(n)}\equiv X$.
Assumption $X_t\in W^{1,p}_\Theta$ for every $t\in[0,T]$ and the $L^p$-bounds on $\Xstar_T$ and $\sup_{t\le T}|D_\Theta X_t|$
imply that Lemma~\ref{lem:criterion}\,\textup{(i)} holds (with $\Xstarn_T=\Xstar_T$).
The right-continuity of $D_\Theta X$ gives \textup{(ii)}, and \textup{(iii)} is immediate.
Finally, the strict positivity implies
$\Pbb(D_\Theta X_t\neq 0\text{ on }\{t\in(0,T]:X_t=\Xstar_T\})=1$.
Hence Lemma~\ref{lem:criterion} yields the claim.
\end{proof}

\begin{remark}
Our proofs of Theorems~\ref{thm:elliptic} and~\ref{thm:jump} consist in constructing an admissible direction $\Theta$
such that $D_\Theta X_t>0$ for all $t>0$.
Once such a direction is available, Proposition~\ref{prop:unified} reduces the absolute-continuity statement to:
(i) $X_t\in W^{1,p}_\Theta$ for all $t$,
(ii) moment bounds for $\Xstar_T$,
(iii) an $L^p$ bound for $\sup_{t\le T}|D_\Theta X_t|$ and a right-continuous version of $D_\Theta X$.
In the elliptic regime, \textup{(i)} is obtained in Lemma~\ref{lem:DX-identification}, while \textup{(ii)} and \textup{(iii)} follow from Lemmas~\ref{lem:apriori} and~\ref{lem:J}, together with the explicit formula in Lemma~\ref{lem:DXpos}.
In the pure-jump regime, the moment bounds are given in Lemma~\ref{lem:jump-moment}, and membership $X\in W^{1,p}_\Theta$ with the explicit representation \eqref{eq:DXjump}
follows from the truncation argument recorded in Lemma~\ref{lem:jump-differentiability} and the explicit representation in Lemma~\ref{lem:DX-jump-expression}.
\end{remark}

\section{Main results}\label{sec:main}

\subsection{Uniformly elliptic Brownian part}

We begin with the general jump-diffusion model \eqref{eq:SDE-general} driven by both a Brownian motion and a non-truncated L\'evy process.
The key hypothesis is uniform ellipticity of the diffusion coefficient $\sigma$, which allows us to construct a purely Brownian perturbation
and obtain an explicit strictly positive formula for $D_\Theta X_t$.
For clarity we state the coefficient assumptions separately.

\begin{assumption}[Coefficients: elliptic Brownian regime]\label{ass:elliptic}
Assume the following. Here $\|\cdot\|_\infty$ denotes the supremum norm over the relevant domain.
\begin{enumerate}[label=(E\arabic*),leftmargin=2.7em]
\item $b,\sigma:[0,T]\times\R\to\R$ are measurable in $t$ and $C^1$ in $x$ with
$\|\partial_x b\|_\infty+\|\partial_x\sigma\|_\infty<\infty$ and $\|b(\cdot,0)\|_\infty+\|\sigma(\cdot,0)\|_\infty<\infty$;
\item (\emph{Uniform ellipticity}) there exists $\sigma_0>0$ such that $\inf_{(t,x)\in[0,T]\times\R}|\sigma(t,x)|\ge\sigma_0$;
\item (\emph{Boundedness}) $\|\sigma\|_\infty<\infty$;
\item $\kappa:[0,T]\to\R$ is a bounded \emph{deterministic} measurable function, i.e., $\|\kappa\|_\infty<\infty$.
\end{enumerate}
\end{assumption}

Assumption~\ref{ass:elliptic}\,\textup{(E3)} (boundedness of $\sigma$) is imposed for convenience in the proof of Theorem~\ref{thm:elliptic}.
It will be removed in Corollary~\ref{cor:elliptic-linear} under the weaker linear growth implied by \textup{(E1)}.
We now state the density result for the running maximum.

\begin{theorem}[Density of $\Xstar_T$ in the elliptic Brownian regime]\label{thm:elliptic}
Let Assumption~\ref{ass:Levy} hold and consider \eqref{eq:SDE-general}.
Under Assumption~\ref{ass:elliptic}, for every $T>0$ the law of $\Xstar_T=\sup_{0\le s\le T}X_s$
is absolutely continuous with respect to the Lebesgue measure on $\R$.
\end{theorem}

\begin{remark}[Relation to Song--Xie \cite{SongXie2018}]
Song and Xie \cite{SongXie2018} established a running-maximum criterion on the Wiener--Poisson space and verified it for
L\'evy--It\^o diffusions driven by a Brownian motion and a \emph{truncated} symmetric stable process.
Theorem~\ref{thm:elliptic} shows that, under uniform ellipticity of the Brownian coefficient, one can construct a strictly positive admissible direction
using only the Brownian component.
In particular, we obtain absolute continuity of $\Xstar_T$ for the \emph{original} non-truncated model without any jump-direction nondegeneracy.
\end{remark}

\begin{remark}[Example: an elliptic jump diffusion]
Let $L$ be a symmetric $\alpha$-stable process with $\alpha\in(1,2)$.
Consider, for instance,
\[
\dd X_t=-\theta X_t\,\dd t+\sigma(X_t)\,\dd W_t+\dd L_t,\qquad X_0=x,
\]
with $\theta>0$ and $\sigma(x):=\sigma_0+\frac{1}{1+x^2}$.
Then $\sigma$ is $C^1$, bounded, and uniformly elliptic ($\inf_x\sigma(x)=\sigma_0$), while $b(t,x)=-\theta x$ is globally Lipschitz.
Hence Assumption~\ref{ass:elliptic} holds (with $\kappa\equiv1$) and Theorem~\ref{thm:elliptic} applies.
\end{remark}

\begin{remark}[On density regularity]
Theorems~\ref{thm:elliptic} and~\ref{thm:jump} establish \emph{existence} of a density for $\Xstar_T$.
Questions of regularity (continuity, smoothness) are considerably more delicate for running maxima and are not pursued here.
We note that the explicit lower bound
\[
D_\Theta X_t\ge \sigma_0\,J_t\int_0^t J_s^{-2}\,\dd s
\]
from Lemma~\ref{lem:DXpos} below (Section~\ref{sec:proof-elliptic}), where $J$ denotes the Jacobian process introduced in \eqref{eq:J}, may serve as a starting point for Malliavin--Sobolev regularity arguments.
\end{remark}

\begin{corollary}[Dropping boundedness of $\sigma$]\label{cor:elliptic-linear}
Let Assumption~\ref{ass:Levy} hold and consider \eqref{eq:SDE-general}.
Assume (E1), (E2) and (E4) of Assumption~\ref{ass:elliptic}, but drop the boundedness condition (E3).
That is, we assume all conditions of Assumption~\ref{ass:elliptic} except boundedness of $\sigma$; under (E1) this implies a linear growth condition $|\sigma(t,x)|\le C(1+|x|)$.
Then, for every $T>0$, the law of $\Xstar_T$ is absolutely continuous with respect to the Lebesgue measure on $\R$.
\end{corollary}

\begin{proof}
Under (E1) we have $|\sigma(t,x)|\le \sup_{t}|\sigma(t,0)|+\|\partial_x\sigma\|_\infty|x|$, so $\sigma$ has linear growth.
The moment bound Lemma~\ref{lem:apriori-linear} replaces Lemma~\ref{lem:apriori}.
Moreover, Lemma~\ref{lem:DX-identification} provides the existence of the directional derivative $D_\Theta X$ (and the linear SDE it satisfies) without assuming boundedness of $\sigma$.

It remains to verify the $L^p$-integrability of $\sup_{t\le T}|D_\Theta X_t|$ under the weaker growth assumption on $\sigma$.
Indeed, by the representation \eqref{eq:DX-formula} from Lemma~\ref{lem:DXpos},
\[
\sup_{t\le T}|D_\Theta X_t|
\le T\,\sup_{s\le T}|\sigma(s,X_s)|\ \sup_{t\le T}J_t\ \sup_{s\le T}J_s^{-2}.
\]
Fix $p\in(1,\beta)$ and choose $q$ such that $p<q<\beta$.
Using the linear growth of $\sigma$ and Lemma~\ref{lem:apriori-linear}, we have $\sup_{s\le T}|\sigma(s,X_s)|\in L^q(\Omega)$.
On the other hand, Lemma~\ref{lem:J} yields $\E[\sup_{t\le T}J_t^{r}]<\infty$ and $\E[\sup_{t\le T}J_t^{-r}]<\infty$ for every $r\ge1$,
since $\|\partial_x b\|_\infty+\|\partial_x\sigma\|_\infty<\infty$.
We apply H\"older's inequality with three factors: choose exponents $r_1,r_2,r_3>1$ satisfying $1/r_1+1/r_2+1/r_3=1$ with $r_1=q/p$ (so that $\sup|\sigma|^p\in L^{r_1}$), and $r_2,r_3$ large enough that $\sup J_t^{p}\in L^{r_2}$ and $\sup J_s^{-2p}\in L^{r_3}$ (which holds for any choice by Lemma~\ref{lem:J}). For instance, one may take $r_1=q/p$ and $r_2=r_3=2q/(q-p)$.
This gives $\E[\sup_{t\le T}|D_\Theta X_t|^p]<\infty$.
The rest of the proof of Theorem~\ref{thm:elliptic} is unchanged.
\end{proof}

\subsection{Pure-jump regime with time-dependent jump weight}

We state a complementary result where the Brownian part is absent (or degenerate) and the jump weight may vanish on subintervals.

\begin{assumption}[Coefficients: pure-jump regime]\label{ass:jump}
Assume:
\begin{enumerate}[label=(J\arabic*),leftmargin=2.7em]
\item $\sigma\equiv 0$ and $b:[0,T]\times\R\to\R$ is measurable such that $x\mapsto b(t,x)$ is $C^1$ for every $t\in[0,T]$, with $\|\partial_x b\|_\infty<\infty$ and $\|b(\cdot,0)\|_\infty<\infty$;
\item $\kappa:[0,T]\to\R$ is a bounded \emph{deterministic} measurable function and satisfies the nondegeneracy-in-time condition
\begin{equation}\label{eq:NDtime}
\int_0^t \kappa(s)^2\,\dd s>0,\qquad \forall t\in(0,T].
\end{equation}
\end{enumerate}
\end{assumption}

Condition~\eqref{eq:NDtime} rules out a completely silent initial interval for the jump noise.
Theorem~\ref{thm:jump} below shows that this minimal nondegeneracy in time, together with infinite activity of $\nu$, suffices to guarantee a density for $\Xstar_T$.

\begin{remark}[Time-dependent drift in the pure-jump regime]
Unlike \cite{NakagawaSuzuki2024}, Theorem~\ref{thm:jump} allows the drift to be time-dependent.
The proof is unchanged except that $b'(X_s)$ is replaced by $\partial_x b(s,X_s)$ in the linearized equation for $D_\Theta X$
and in the associated Jacobian $\exp(\int_0^t \partial_x b(s,X_s)\,\dd s)$; our choice of direction $v$ incorporates this Jacobian.
The time-independent case $b(t,x)\equiv b(x)$ is included.
\end{remark}

\begin{theorem}[Density of $\Xstar_T$ in a pure-jump model with a time-dependent weight (possibly vanishing)]\label{thm:jump}
Let Assumption~\ref{ass:Levy} hold and consider
\begin{equation}\label{eq:SDE-jump}
\dd X_t=b(t,X_t)\dt+\kappa(t)\,\dd L_t,\qquad X_0=x.
\end{equation}
Under Assumption~\ref{ass:jump}, for every $T>0$ the law of $\Xstar_T$ is absolutely continuous with respect to the Lebesgue measure on $\R$.
\end{theorem}

\begin{remark}[Relation to the constant-weight case and sharpness]
When $\kappa(t)\equiv \sigma_2\neq 0$ and the drift is time independent (i.e., $b(t,x)\equiv b(x)$), \eqref{eq:SDE-jump} reduces to the additive-coefficient model studied in \cite{NakagawaSuzuki2024}
(with $\sigma_1=0$), and Theorem~\ref{thm:jump} recovers their absolute continuity result in the pure-jump case.
The new feature here is that $\kappa$ may be time-dependent and may vanish on subintervals, as long as the minimal
nondegeneracy-in-time condition \eqref{eq:NDtime} holds.
Theorem~\ref{thm:silent} shows that if the noise is silent on an initial interval, atoms may appear in $\law(\Xstar_T)$,
so \eqref{eq:NDtime} is essentially sharp.
We emphasize that the obstruction arises specifically from \emph{initial} silence: if $\kappa$ vanishes on an interior interval $[a,b]\subset(0,T]$ but condition~\eqref{eq:NDtime} still holds (because $\int_0^t\kappa(s)^2\,\dd s>0$ for every $t>0$), then Theorem~\ref{thm:jump} applies and $\Xstar_T$ admits a density.
\end{remark}

\begin{remark}[Example: a weight vanishing at the origin]
Let $L$ be a symmetric $\alpha$-stable process with $\alpha\in(1,2)$ and consider
\[
\dd X_t=b(X_t)\,\dd t+t^\gamma\,\dd L_t,\qquad X_0=x,
\]
with $\gamma>0$ and $b\in C^1$ with bounded derivative.
Here $\kappa(t)=t^\gamma$ vanishes at $t=0$, but the nondegeneracy-in-time condition holds since
$\int_0^t s^{2\gamma}\,\dd s>0$ for every $t>0$.
Therefore Theorem~\ref{thm:jump} applies.
\end{remark}

\begin{remark}[A mixed-direction extension]
Proposition~\ref{prop:unified} also permits mixed directions $\Theta=(h,v)$.
This suggests that one may treat models in which both $\sigma$ and $\kappa$ are present but each may degenerate on subintervals, under an appropriate combined nondegeneracy-in-time condition.
The main difficulty is that the variation-of-constants formula then involves a \emph{sum} of a Lebesgue integral and a Poisson integral, and guaranteeing pathwise strict positivity of this sum seems to require new ideas.
We leave such mixed-degeneracy regimes for future work.
\end{remark}

\subsection{Silent intervals and an atom--density decomposition}

The next result shows that the nondegeneracy-in-time condition is essentially sharp.

\begin{theorem}[Silent intervals create atoms]\label{thm:silent}
Assume that $\kappa:[0,T]\to\R$ is a deterministic measurable function, and that there exists $\delta\in(0,T)$ such that $\sigma(t,x)=0$ for a.e.\ $t\in[0,\delta]$ and all $x$,
and $\kappa(t)=0$ for a.e.\ $t\in[0,\delta]$.
Assume moreover that the drift coefficient is globally Lipschitz in the state variable (for instance, this holds under Assumption~\ref{ass:elliptic}\,\textup{(E1)} or Assumption~\ref{ass:jump}\,\textup{(J1)}).
Then $X$ is deterministic on $[0,\delta]$ and $\Xstar_\delta$ is a deterministic constant $m_0$.
Moreover,
\[
\Xstar_T=\max\Big(m_0,\ \sup_{\delta\le s\le T}X_s\Big),
\]
and hence $\law(\Xstar_T)$ may contain an atom at $m_0$:
\[
\Pbb(\Xstar_T=m_0)=\Pbb\Big(\sup_{\delta\le s\le T}X_s\le m_0\Big).
\]
If, in addition, the shifted post-$\delta$ dynamics satisfies the assumptions of Theorem~\ref{thm:elliptic} (or Theorem~\ref{thm:jump}),
then the restriction of $\law(\Xstar_T)$ to $(m_0,\infty)$ admits a density (an explicit atom--density decomposition).
\end{theorem}

\begin{remark}[On the a.e.\ formulation in Theorem~\ref{thm:silent}]
The a.e.\ formulation in Theorem~\ref{thm:silent} is sufficient, since the stochastic integrals are unchanged by modifications on $\dd t$-null sets
(and on $\dd t\otimes\nu$-null sets for the jump term); see the proof of Theorem~\ref{thm:silent} for details.
\end{remark}

\begin{remark}[Explicit atom--density decomposition]
In the setting of Theorem~\ref{thm:silent}, let $Y:=\sup_{\delta\le s\le T}X_s$.
Since $\Xstar_T=\max(m_0,Y)\ge m_0$ a.s., the law of $\Xstar_T$ is supported on $[m_0,\infty)$.
If $Y$ admits a density $f_Y$ (for instance, when the shifted post-$\delta$ dynamics satisfies the assumptions of
Theorem~\ref{thm:elliptic} or Theorem~\ref{thm:jump}), then
\[
\law(\Xstar_T)=\Pbb(Y\le m_0)\,\delta_{m_0}+ f_Y(x)\,\1_{(m_0,\infty)}(x)\dx.
\]
Here $f_Y$ denotes the Lebesgue density of $\law(Y)$ and $\delta_{m_0}$ denotes the Dirac measure at $m_0$.
(Note that $Y$ admits a density, so $\Pbb(Y=m_0)=0$, and the decomposition is well defined.)
In particular, the absolutely continuous component of $\law(\Xstar_T)$ on $(m_0,\infty)$ coincides with the law of $Y$
restricted to $(m_0,\infty)$.
\end{remark}

\begin{remark}[Example: a canonical silent interval]
In the pure-jump setting, a canonical example of a silent interval is $\kappa(t)=\1_{[\delta,T]}(t)$ and $\sigma\equiv0$.
Then $X$ evolves deterministically on $[0,\delta]$ and the running maximum satisfies
$\Xstar_T=\max\big(m_0,\sup_{\delta\le s\le T}X_s\big)$, so that the atomic mass is
\[
\Pbb(\Xstar_T=m_0)=\Pbb\Big(\sup_{\delta\le s\le T}X_s\le m_0\Big).
\]
\end{remark}

\begin{remark}[Shifted post-$\delta$ dynamics]
In Theorem~\ref{thm:silent}, the phrase ``shifted post-$\delta$ dynamics'' can be made explicit as follows.
Define $\tilde X_t:=X_{\delta+t}$ for $t\in[0,T-\delta]$, and let $\tilde W_t:=W_{\delta+t}-W_\delta$.
Moreover, define the shifted Poisson random measure $N^\delta$ by
\[
N^\delta((0,t]\times A):=N((\delta,\delta+t]\times A),\qquad t\in[0,T-\delta],\ A\subset\R_0\ \text{Borel}.
\]
Then $(\tilde W,N^\delta)$ is again a Wiener--Poisson pair and $\tilde X$ solves an SDE of the form \eqref{eq:SDE-general} on $[0,T-\delta]$
with shifted coefficients $\tilde b(t,x):=b(\delta+t,x)$, $\tilde\sigma(t,x):=\sigma(\delta+t,x)$ and $\tilde\kappa(t):=\kappa(\delta+t)$.
Accordingly, the nondegeneracy-in-time condition becomes
$\int_0^t \tilde\kappa(s)^2\,\dd s=\int_\delta^{\delta+t}\kappa(s)^2\,\dd s>0$ for all $t\in(0,T-\delta]$,
and the uniform ellipticity condition becomes $\inf_{(t,x)\in[\delta,T]\times\R}|\sigma(t,x)|\ge\sigma_0$.
\end{remark}

\section{Proof of Theorem~\ref{thm:elliptic}}\label{sec:proof-elliptic}

We verify Proposition~\ref{prop:unified} with a Brownian direction $\Theta=(h,0)$.
The proof proceeds as follows:
(i) we establish moment bounds for $X$ and its Jacobian $J$ (Lemmas~\ref{lem:apriori} and~\ref{lem:J});
(ii) we identify the Brownian directional derivative $D_\Theta X$ in a self-contained way (Lemma~\ref{lem:DX-identification});
(iii) we choose $h$ so that $D_\Theta X_t$ admits an explicit strictly positive representation (Lemma~\ref{lem:DXpos});
(iv) we check the $L^p$ bounds and right-continuity required in Proposition~\ref{prop:unified}.

\subsection{Existence, moment bounds, and the Jacobian}

We first collect the basic $L^p$ estimates needed to apply the running-maximum criterion.
In particular, we need moment bounds for $X$ to control $\Xstar_T$ and we need moment bounds for the Jacobian $J$,
which will enter both the choice of the direction $h$ and the explicit representation of $D_\Theta X$.

Under Assumption~\ref{ass:elliptic}, \eqref{eq:SDE-general} has a unique strong solution.
Fix $p\in(1,\beta)$ from Assumption~\ref{ass:Levy}.
By Lemma~\ref{lem:apriori}, $\E[\sup_{t\le T}|X_t|^p]<\infty$.

Define the first variation (Jacobian) process $J$ by
\begin{equation}\label{eq:J}
\dd J_t=\partial_x b(t,X_t)J_t\dt+\partial_x\sigma(t,X_t)J_t\,\dd W_t,\qquad J_0=1.
\end{equation}
Note that $J$ is continuous since it is driven only by $dt$ and $dW$, although $X$ itself is c\`adl\`ag.

\begin{lemma}[Positivity and moments of $J$]\label{lem:J}
Under Assumption~\ref{ass:elliptic}, $J_t>0$ a.s.\ for all $t$, and for any $q\ge 1$,
\[
\E\Big[\sup_{0\le t\le T}J_t^{q}\Big]+\E\Big[\sup_{0\le t\le T}J_t^{-q}\Big]<\infty.
\]
\end{lemma}

\begin{proof}
Set
\[
c_t:=\partial_x\sigma(t,X_t),\qquad a_t:=\partial_x b(t,X_t).
\]
By Assumption~\ref{ass:elliptic}, $|c_t|\le \|\partial_x\sigma\|_\infty=:A$ and $|a_t|\le \|\partial_x b\|_\infty=:B$.
The linear SDE \eqref{eq:J} admits the stochastic exponential representation
\begin{equation}\label{eq:Jexp}
J_t=\exp\Big(\int_0^t c_s\,\dd W_s+\int_0^t\big(a_s-\tfrac12 c_s^2\big)\,\dd s\Big),\qquad t\in[0,T].
\end{equation}
In particular, $J_t>0$ a.s.

Fix $q\ge 1$. From \eqref{eq:Jexp} we can write
\begin{align*}
J_t^{q}
&=\underbrace{\exp\Big(q\int_0^t c_s\,\dd W_s-\tfrac12 q^2\int_0^t c_s^2\,\dd s\Big)}_{=:\,M_t^{+}}
\cdot\exp\Big(\int_0^t q a_s\,\dd s+\tfrac12 q(q-1)\int_0^t c_s^2\,\dd s\Big)\\
&\le M_t^{+}\exp\Big(qBT+\tfrac12 q(q-1)A^2T\Big),\\[4pt]
J_t^{-q}
&=\underbrace{\exp\Big(-q\int_0^t c_s\,\dd W_s-\tfrac12 q^2\int_0^t c_s^2\,\dd s\Big)}_{=:\,M_t^{-}}
\cdot\exp\Big(\int_0^t (-q) a_s\,\dd s+\tfrac12 q(q+1)\int_0^t c_s^2\,\dd s\Big)\\
&\le M_t^{-}\exp\Big(qBT+\tfrac12 q(q+1)A^2T\Big).
\end{align*}
Since $|c_s|\le A$, the Novikov condition
$\E\big[\exp(\frac12 q^2\int_0^T c_s^2\,\dd s)\big]\le \exp(\frac12 q^2 A^2 T)<\infty$
holds, so $M^{\pm}$ are true martingales.
Since $M^{\pm}$ are nonnegative martingales, $(M^{\pm})^2$ are submartingales.
By Doob's $L^2$ maximal inequality,
\[
\E\Big[\sup_{0\le t\le T}(M_t^{+})^2\Big]\le 4\,\E[(M_T^{+})^2],\qquad
\E\Big[\sup_{0\le t\le T}(M_t^{-})^2\Big]\le 4\,\E[(M_T^{-})^2].
\]
Using the Cauchy--Schwarz inequality,
\[
\E\Big[\sup_{0\le t\le T}M_t^{\pm}\Big]
\le \E\Big[\sup_{0\le t\le T}(M_t^{\pm})^2\Big]^{1/2}
\le 2\,\E[(M_T^{\pm})^2]^{1/2}.
\]
Finally, we bound $\E[(M_T^{+})^2]$.
Here $\mathcal{E}(\cdot)$ denotes the Dol\'eans-Dade stochastic exponential.
Writing $(M_T^{+})^2=\mathcal{E}(2q\,c\cdot W)_T\cdot\exp(q^2\int_0^T c_s^2\,\dd s)$
and using $|c_s|\le A$, we obtain
\[
(M_T^{+})^2\le \mathcal{E}(2q\,c\cdot W)_T\cdot\exp(q^2 A^2 T).
\]
Since $|c_s|\le A$, the Novikov condition
$\E[\exp(2q^2\int_0^T c_s^2\,\dd s)]\le \exp(2q^2 A^2 T)<\infty$
holds, so $\mathcal{E}(2q\,c\cdot W)$ is a true martingale with $\E[\mathcal{E}(2q\,c\cdot W)_T]=1$.
Therefore
\[
\E[(M_T^{+})^2]\le \exp(q^2 A^2 T)<\infty,
\]
and the same bound holds for $\E[(M_T^{-})^2]$.
Hence $\E[\sup_{t\le T}M_t^{\pm}]<\infty$.

Since $J_t^q\le M_t^+\cdot C_1$ and $J_t^{-q}\le M_t^-\cdot C_2$ with deterministic constants
$C_1:=\exp(qBT+\frac12 q(q-1)A^2T)$ and $C_2:=\exp(qBT+\frac12 q(q+1)A^2T)$,
\[
\E\Big[\sup_{0\le t\le T}J_t^{q}\Big]
\le C_1\,\E\Big[\sup_{0\le t\le T}M_t^{+}\Big]<\infty,
\qquad
\E\Big[\sup_{0\le t\le T}J_t^{-q}\Big]
\le C_2\,\E\Big[\sup_{0\le t\le T}M_t^{-}\Big]<\infty.
\]
This proves the claim.
\end{proof}

\subsection{A strictly positive Brownian direction}

We now choose an adapted Brownian direction $h$ that aligns with the sign of $\sigma(t,X_t)$ and compensates the Jacobian.
This makes the forcing term $\sigma(t,X_t)h(t)$ in the linearized equation nonnegative and yields a strictly positive representation
of $D_\Theta X_t$ through a variation-of-constants argument.

Let $\Theta=(h,0)$ and choose
\begin{equation}\label{eq:hchoice}
h(t):=\frac{\sigma(t,X_t)}{|\sigma(t,X_t)|}\,J_t^{-1},\qquad t\in[0,T].
\end{equation}

\medskip
For later reference, we record the linearized equation associated with a Brownian direction $\Theta=(h,0)$.
When the directional derivative exists, it satisfies
\begin{equation}\label{eq:linearized}
\begin{aligned}
D_\Theta X_t
&=\int_0^t \partial_x b(s,X_s)\,D_\Theta X_s\,\dd s
+\int_0^t \partial_x\sigma(s,X_s)\,D_\Theta X_s\,\dd W_s\\
&\quad+\int_0^t \sigma(s,X_s)h(s)\,\dd s,\qquad t\in[0,T].
\end{aligned}
\end{equation}
We will justify \eqref{eq:linearized} in Lemma~\ref{lem:DX-identification}.

\begin{lemma}[Directional derivative and positivity]\label{lem:DXpos}
With $\Theta$ as above, $h\in \Hb_{\infty-}$ and the directional derivative $D_\Theta X_t$ exists, admits
a continuous version, and satisfies
\[
\Pbb\big(D_\Theta X_t>0,\ \forall t\in(0,T]\big)=1.
\]
\end{lemma}

\begin{proof}
We split the argument into four steps.

\medskip
\noindent\emph{Step 1: $h\in \Hb_{\infty-}$.}
By uniform ellipticity, $|\sigma(t,X_t)|\ge \sigma_0>0$ for all $t$, so $|h(t)|=J_t^{-1}$.
For any $r\ge 1$ we therefore have
\[
\E\Big[\Big(\int_0^T |h(s)|^2\,\dd s\Big)^{r/2}\Big]
=\E\Big[\Big(\int_0^T J_s^{-2}\,\dd s\Big)^{r/2}\Big]
\le T^{r/2}\E\Big[\sup_{0\le s\le T}J_s^{-r}\Big]<\infty
\]
by Lemma~\ref{lem:J}. Hence $h\in \Hb_r$ for all $r\ge 1$, i.e., $h\in \Hb_{\infty-}$.

\medskip
\noindent\emph{Step 2: existence of $D_\Theta X$ and the linearized equation.}
Fix $p\in(1,\beta)$.
Since $h\in \Hb_{\infty-}$ by Step~1, we may apply Lemma~\ref{lem:DX-identification} with $\Theta=(h,0)$.
It yields that for every $t\in[0,T]$ we have $X_t\in W^{1,p}_\Theta$ and that $D_\Theta X$ is the unique adapted solution to the linearized equation~\eqref{eq:linearized}.
In particular, $D_\Theta X$ admits a continuous version.

\medskip
\noindent\emph{Step 3: variation of constants and positivity.}
Let $J$ be the Jacobian defined by \eqref{eq:J}. Recall from Step~2 that $D_\Theta X$ has continuous paths; since $J$ is also continuous, $J^{-1}D_\Theta X$ is a continuous semimartingale.
Applying It\^o's formula to $J_t^{-1}D_\Theta X_t$ (equivalently, using the standard variation-of-constants formula for linear SDEs), we obtain from \eqref{eq:linearized} that
\begin{equation}\label{eq:DX-formula}
D_\Theta X_t
=J_t\int_0^t J_s^{-1}\sigma(s,X_s)h(s)\,\dd s
=J_t\int_0^t |\sigma(s,X_s)|\,J_s^{-2}\,\dd s,\qquad t\in[0,T].
\end{equation}
By Lemma~\ref{lem:J}, $J_t>0$ for all $t$ a.s., and by ellipticity $|\sigma(s,X_s)|\ge\sigma_0$.
Therefore $D_\Theta X_t>0$ for all $t\in(0,T]$ a.s.

\medskip
\noindent\emph{Step 4: $L^p$-integrability of $D_\Theta X$.}
Under (E3), $\|\sigma\|_\infty<\infty$, so by the representation above,
\[
\sup_{0\le t\le T}|D_\Theta X_t|
\le \|\sigma\|_\infty\;T\;\sup_{0\le t\le T}J_t\;\sup_{0\le s\le T}J_s^{-2}.
\]
Lemma~\ref{lem:J} gives finite moments of $\sup_{t\le T}J_t$ and $\sup_{t\le T}J_t^{-1}$ of all orders, and thus
\[
\E\Big[\sup_{0\le t\le T}|D_\Theta X_t|^p\Big]<\infty,\qquad \forall\,p\ge1.
\]
This completes the proof.
\end{proof}

\subsection{Conclusion via the running-maximum criterion}

Having established strict positivity of $D_\Theta X_t$ and the required $L^p$ bounds, we can now apply the criterion with the constant approximation $X^{(n)}\equiv X$.

\begin{proof}[Proof of Theorem~\ref{thm:elliptic}]
Fix $p\in(1,\beta)$.
By Lemma~\ref{lem:apriori} we have $\E[|\Xstar_T|^p]\le \E[\sup_{t\le T}|X_t|^p]<\infty$.
With $\Theta=(h,0)$ from \eqref{eq:hchoice}, Lemma~\ref{lem:DXpos} and Lemma~\ref{lem:DX-identification} yield that $X_t\in W^{1,p}_\Theta$
for all $t\in[0,T]$ and that $D_\Theta X$ admits a continuous version solving \eqref{eq:linearized}.
Moreover, Step~4 of the proof of Lemma~\ref{lem:DXpos} shows that $\E[\sup_{t\le T}|D_\Theta X_t|^p]<\infty$.
Thus conditions (i)--(iii) of Lemma~\ref{lem:criterion} hold with the constant approximation sequence $X^{(n)}\equiv X$.
Finally, Lemma~\ref{lem:DXpos} implies that $D_\Theta X_t>0$ for all $t\in(0,T]$ a.s., hence
\[
\Pbb\big(D_\Theta X_t\neq 0 \text{ on }\{t\in(0,T]: X_t=\Xstar_T\}\big)=1.
\]
Therefore Lemma~\ref{lem:criterion} applies and $\law(\Xstar_T)$ is absolutely continuous.
\end{proof}

\section{Proof of Theorem~\ref{thm:jump}}

We verify Proposition~\ref{prop:unified} in a purely jump direction $\Theta=(0,v)$.
Compared with \cite{NakagawaSuzuki2024}, the main difference is that the jump weight $\kappa(t)$ may vary in time and may vanish on subintervals.
Our argument has three main steps:
(i) we construct an admissible jump direction $v$ and obtain an explicit representation of $D_\Theta X$ (Lemmas~\ref{lem:v-admissible} and~\ref{lem:DX-jump-expression});
(ii) we prove a weighted-Poisson strict-positivity lemma which yields $D_\Theta X_t>0$ for all $t>0$ (Lemma~\ref{lem:poisson-positive});
(iii) we establish the moment and Sobolev bounds needed to apply Lemma~\ref{lem:criterion} (Lemmas~\ref{lem:jump-moment} and~\ref{lem:jump-differentiability}).

\subsection{Choice of direction and expression for the directional derivative}

The construction of $v$ has two objectives.
First, $v$ must be admissible in the sense of $\Vb_{\infty-}$, which dictates integrability and differentiability conditions in the jump-size variable $z$.
Second, we want the product $\kappa(t)v(t,z)$ appearing in the linearized equation for $D_\Theta X$ to be nonnegative and to vanish only when $\kappa(t)=0$ or $\eta(z)=0$.
The deterministic cut-off $\eta$ in \eqref{eq:eta-choice} achieves both.

Let $\eta:\R\to[0,\infty)$ be a $C^1$ function supported in $\{|z|\le 1\}$ such that
\begin{equation}\label{eq:eta-choice}
\eta(z)=|z|^4\ \text{ for }\ |z|\le \tfrac12,\qquad \eta(z)>0\ \text{ for all }\ 0<|z|<1.
\end{equation}

\begin{remark}[On the choice of $\eta$]
The concrete choice \eqref{eq:eta-choice} is made for convenience.
In the definition of the admissible space $\Vb_p$ one may take $\Gamma=(-2,2)$, so that on $\{|z|\le 1\}$ the weight becomes
$\varrho(z)=1\vee |z|^{-1}$ (see Lemma~\ref{lem:v-admissible} below).
Here and below, $L^p(\nu)$ stands for $L^p(\R_0,\nu)$. The bound $\eta(z)=O(|z|^4)$ near $0$ then guarantees that $\eta'\in L^p(\nu)$ and $\eta\,\varrho\in L^p(\nu)$ for every $p>1$
under Assumption~\ref{ass:Levy}\,(L1).
More generally, any nonnegative $C^1$ function $\eta$ supported in $\{|z|\le 1\}$ such that
$\nu(\{z:\eta(z)>0\})=\infty$ and
\[
\int_{\R_0}\big(|\eta'(z)|^p+|\eta(z)|^p\varrho(z)^p\big)\,\nu(\dd z)<\infty
\]
(for the chosen $p>1$) can be used in the same construction.
For instance, any smooth cut-off of $|z|^k$ near $0$ with $k\ge 3$ meets these requirements under (L1).
\end{remark}

For $t\in[0,T]$ define
\[
v(t,z):=\kappa(t)\exp\Big(-\int_0^t \partial_x b(s,X_s)\,\dd s\Big)\eta(z),
\qquad \Theta:=(0,v).
\]
\begin{lemma}[Admissibility of the jump direction]\label{lem:v-admissible}
Under Assumptions~\ref{ass:Levy} and~\ref{ass:jump}, the function $v$ defined above belongs to $\Vb_{\infty-}$.
\end{lemma}

\begin{proof}
The field $v$ is predictable since $t\mapsto \int_0^t \partial_x b(s,X_s)\,\dd s$ is adapted and continuous and $\eta$ is deterministic.
Recall from \cite[Section~3]{NakagawaSuzuki2024} that, in dimension one and for $p>1$,
\[
\|v\|_{\Vb_p}=\|\partial_z v\|_{L_p^1}+\|v\,\varrho\|_{L_p^1},
\]
where for a predictable field $f$ we set
\[
\|f\|_{L_p^1}
:=\Big(\E\Big[\Big(\int_0^T\!\int_{\R_0} |f(t,z)|\,\nu(\dd z)\,\dd t\Big)^{p}\Big]\Big)^{1/p}
+\Big(\E\Big[\int_0^T\!\int_{\R_0} |f(t,z)|^p\,\nu(\dd z)\,\dd t\Big]\Big)^{1/p}.
\]
The weight is given by $\varrho(z)=1\vee d(z,\Gamma_0^{\,c})^{-1}$ for some open set $\Gamma\subset\R$ containing the origin,
where $d(z,B):=\inf_{y\in B}|z-y|$ denotes the Euclidean distance and $\Gamma_0:=\Gamma\setminus\{0\}$. We may take $\Gamma=(-2,2)$, so that on $\{|z|\le 1\}$ we have
$\varrho(z)=1\vee |z|^{-1}$.

Set
\[
A_t:=\kappa(t)\exp\Big(-\int_0^t \partial_x b(s,X_s)\,\dd s\Big),\qquad t\in[0,T].
\]
By Assumption~\ref{ass:jump}, $\sup_{t\le T}|A_t|\le C_A$ for a deterministic constant $C_A<\infty$.
Since $v(t,z)=A_t\eta(z)$ and $\partial_z v(t,z)=A_t\eta'(z)$, it suffices to check the integrability of
$\eta'$ and $\eta\,\varrho$ against $\nu$.

By \eqref{eq:eta-choice}, for $|z|\le \tfrac12$,
\[
|\eta(z)|\,\varrho(z)\le |z|^4\cdot |z|^{-1}=|z|^3,\qquad |\eta'(z)|\le C|z|^3,
\]
while on $\tfrac12\le |z|\le 1$ the functions $\eta$, $\eta'$ and $\varrho$ are bounded.
Assumption~\ref{ass:Levy}\,(L1) implies $\int_{|z|\le 1}|z|^2\,\nu(\dd z)<\infty$ and also $\nu(\{|z|\ge\tfrac12\})<\infty$ (indeed, $1\wedge |z|^2\ge 1/4$ on $\{|z|\ge 1/2\}$, hence $\nu(\{|z|\ge 1/2\})\le 4\int_{\R_0}(1\wedge |z|^2)\,\nu(\dd z)<\infty$).
Hence,
\[
\int_{\R_0}\big(|\eta'(z)|+|\eta(z)|\,\varrho(z)\big)\,\nu(\dd z)<\infty,
\qquad
\int_{\R_0}\big(|\eta'(z)|^p+|\eta(z)|^p\,\varrho(z)^p\big)\,\nu(\dd z)<\infty,
\]
since on $|z|\le 1$ we have $|z|^{3p}\le |z|^2$ for every $p>1$ (since $3p>2$).
Therefore $\partial_z v\in L_p^1$ and $v\,\varrho\in L_p^1$, i.e., $v\in \Vb_p$ for every $p>1$.
Equivalently, $v\in \Vb_{\infty-}$.
\end{proof}

\begin{lemma}[Directional differentiability of \eqref{eq:SDE-jump} with a bounded time-dependent weight]\label{lem:jump-differentiability}
Assume that Assumptions~\ref{ass:Levy} and~\ref{ass:jump} hold, and fix $p\in(1,\beta)$.
Let $\Theta=(0,v)$ with $v$ defined above.
Then, for every $t\in[0,T]$, $X_t\in W^{1,p}_\Theta$ and the directional derivative $D_\Theta X$ admits a c\`adl\`ag version satisfying
\begin{equation}\label{eq:DXjump-linear}
D_\Theta X_t
=\int_0^t \partial_x b(s,X_s)\,D_\Theta X_s\,\dd s
+\int_0^t\int_{\R_0}\kappa(s)\,v(s,z)\,N(\dd s,\dd z),\qquad D_\Theta X_0=0.
\end{equation}
Moreover, in all $L^p$-estimates in the proof, the coefficient $\kappa$ enters only through $\|\kappa\|_\infty$.
\end{lemma}

\begin{proof}
We adapt the large-jump truncation argument of \cite[Section~5]{NakagawaSuzuki2024} and record the only place where the constancy of the jump coefficient is used.

\medskip
\noindent\emph{Step 1: truncation of large jumps.}
For $n\ge1$ define the truncated L\'evy process
\[
L^{(n)}_t:=\int_0^t\int_{|z|\le1} z\,\tilde N(\dd s,\dd z)+\int_0^t\int_{1<|z|\le n} z\,N(\dd s,\dd z),
\qquad t\in[0,T],
\]
and let $X^{(n)}$ be the unique strong solution of
\begin{equation}\label{eq:SDE-jump-trunc}
\dd X^{(n)}_t=b(t,X^{(n)}_t)\dt+\kappa(t)\,\dd L^{(n)}_t,\qquad X^{(n)}_0=x.
\end{equation}
Since $b$ is globally Lipschitz ($\|\partial_x b\|_\infty<\infty$) and $\kappa$ is bounded, these solutions exist and are c\`adl\`ag.

\medskip
\noindent\emph{Step 2: Malliavin differentiability for each truncated model.}
Since $L^{(n)}$ has jump sizes bounded by $n$ and only finitely many jumps on $[0,T]$ with $|z|>1$,
the proof of \cite[Lemma~5.1]{NakagawaSuzuki2024} (which itself builds on the truncation argument of \cite[Lemma~4.3]{SongXie2018}) applies for our fixed $p\in(1,\beta)$ and yields,
for each $n\ge1$ and every $\Theta\in \Hb_{\infty-}\times \Vb_{\infty-}$,
that $X^{(n)}_t\in W^{1,p}_\Theta$ and
\begin{equation}\label{eq:DXjump-n}
D_\Theta X^{(n)}_t
=\int_0^t \partial_x b(s,X^{(n)}_s)\,D_\Theta X^{(n)}_s\,\dd s
+\int_0^t\int_{0<|z|\le n}\kappa(s)\,v(s,z)\,N(\dd s,\dd z),\qquad D_\Theta X^{(n)}_0=0.
\end{equation}
In those arguments the jump coefficient enters only through moment bounds of Poisson integrals.
Therefore, replacing a constant $\sigma_2$ by a bounded deterministic function $\kappa(\cdot)$ only changes the estimates through the uniform bound $|\kappa(s)|\le\|\kappa\|_\infty$.
Moreover, the direction may be random: the proof of \cite[Lemma~5.1]{NakagawaSuzuki2024} requires only that $v$ be predictable and belong to $\Vb_p$.
Our chosen field $v$ is predictable and lies in $\Vb_{\infty-}$ by Lemma~\ref{lem:v-admissible}.

\medskip
\noindent\emph{Step 3: uniform $L^p$-bound for $D_\Theta X^{(n)}$.}
Let $K:=\|\partial_x b\|_\infty$ and set
\[
C_t:=\int_0^t\int_{\R_0}\kappa(s)\,v(s,z)\,N(\dd s,\dd z).
\]
Note that $C_t$ depends on the original solution $X$ (through the direction $v$) but not on the truncation index $n$, since $v(t,\cdot)$ is supported in $\{|z|\le1\}$.
Since $\kappa(s)\,v(s,z)=\kappa(s)^2\exp(-\int_0^s \partial_x b(u,X_u)\,\dd u)\,\eta(z)\ge 0$, we have $C_t\ge 0$ for all $t$.
Moreover, since $v(t,\cdot)$ is supported in $\{|z|\le1\}$, for every $n\ge1$,
\[
\int_0^t\int_{0<|z|\le n}\kappa(s)\,v(s,z)\,N(\dd s,\dd z)=C_t,\qquad t\in[0,T].
\]
Thus \eqref{eq:DXjump-n} reads $D_\Theta X^{(n)}_t=\int_0^t \partial_x b(s,X^{(n)}_s)D_\Theta X^{(n)}_s\,\dd s+C_t$.
Gronwall's inequality gives the pathwise bound
\[
\sup_{0\le t\le T}|D_\Theta X^{(n)}_t|
\le e^{KT}\sup_{0\le t\le T}|C_t|
\le e^{KT}\int_0^T\int_{\R_0}|\kappa(s)|\,|v(s,z)|\,N(\dd s,\dd z).
\]
Since $C_t\ge 0$ is nondecreasing in $t$, we have $\sup_{0\le t\le T}|C_t|=C_T$.
We decompose $C_T$ into its compensator and compensated parts:
\[
C_T=\int_0^T\int_{\R_0}\kappa(s)v(s,z)\,\nu(\dd z)\,\dd s
+\int_0^T\int_{\R_0}\kappa(s)v(s,z)\,\tilde N(\dd s,\dd z).
\]
By the triangle inequality and $(|a|+|b|)^p\le 2^{p-1}(|a|^p+|b|^p)$, it suffices to bound the $p$th moments of the two terms.
Since $\kappa(s)v(s,z)=\kappa(s)^2\exp(-\int_0^s \partial_x b(u,X_u)\,\dd u)\eta(z)$ and $\|\partial_x b\|_\infty<\infty$,
the compensator term is bounded by the deterministic constant
$\|\kappa\|_\infty^2 e^{\|\partial_x b\|_\infty T}T\int \eta(z)\,\nu(\dd z)$.
For the martingale term, set
\[
M_T:=\int_0^T\int \kappa(s)v(s,z)\,\tilde N(\dd s,\dd z).
\]
We treat the cases $p\in(1,2]$ and $p>2$ separately.
If $p\in(1,2]$, then Lyapunov's inequality and the It\^o isometry yield
\[
\E[|M_T|^p]\le \E[|M_T|^2]^{p/2}
=\E\Big[\int_0^T\!\!\int|\kappa(s)v(s,z)|^2\,\nu(\dd z)\,\dd s\Big]^{p/2}.
\]
Using the bound $|\kappa(s)v(s,z)|\le \|\kappa\|_\infty^2e^{\|\partial_x b\|_\infty T}\eta(z)$, we obtain
\[
\E[|M_T|^p]\le \Big(T\|\kappa\|_\infty^4e^{2\|\partial_x b\|_\infty T}\int \eta(z)^2\,\nu(\dd z)\Big)^{p/2}<\infty.
\]
If $p>2$, then the Burkholder--Davis--Gundy (BDG) inequality together with Kunita's first inequality for Poisson integrals (e.g.\ \cite[Theorem~4.4.23]{Applebaum2009}; we shall use these inequalities repeatedly in the sequel) gives
\[
\E[|M_T|^p]\le C_p\Big(\E\Big[\Big(\int_0^T\!\!\int|\kappa(s)v(s,z)|^2\,\nu(\dd z)\,\dd s\Big)^{p/2}\Big]
+\E\Big[\int_0^T\!\!\int|\kappa(s)v(s,z)|^p\,\nu(\dd z)\,\dd s\Big]\Big)<\infty.
\]
In both cases, finiteness follows from boundedness of $\kappa$ and the integrability of $\eta$ and $\eta^p$ under Assumption~\ref{ass:Levy}\,(L1).
Therefore $\E[\sup_{t\le T}|C_t|^p]=\E[C_T^p]<\infty$.
Hence $\sup_{n\in\mathbb{N}}\E[\sup_{t\le T}|D_\Theta X^{(n)}_t|^p]<\infty$,
with constants depending on $\kappa$ only via $\|\kappa\|_\infty$.

\medskip
\noindent\emph{Step 4: convergence $X^{(n)}\to X$.}
Subtracting \eqref{eq:SDE-jump-trunc} from \eqref{eq:SDE-jump}, the difference $\Delta^{(n)}:=X-X^{(n)}$ is driven only by the large-jump part $L-L^{(n)}$.
Write $L_t-L^{(n)}_t=\int_0^t\int_{|z|>n} z\,N(\dd s,\dd z)$, which is a compound Poisson process.
Let $\Delta^{(n)}_t:=X_t-X^{(n)}_t$.
Using the Lipschitz property of $b$ and Gronwall's lemma,
\[
\sup_{0\le t\le T}|\Delta^{(n)}_t|
\le e^{KT}\sup_{0\le t\le T}\Big|\int_0^t\kappa(s)\,\dd(L_s-L^{(n)}_s)\Big|
\le e^{KT}\|\kappa\|_\infty\int_0^T\int_{|z|>n}|z|\,N(\dd s,\dd z).
\]
Since $\nu(\{|z|>n\})<\infty$ for $n\ge1$, we may write
$\int_0^T\int_{|z|>n}|z|\,N(\dd s,\dd z)=\sum_{j=1}^{N_n}|Z_{n,j}|$
with $N_n\sim\mathrm{Poisson}(\lambda_n)$, $\lambda_n:=T\nu(\{|z|>n\})$, and i.i.d.\ jumps $Z_{n,j}$ with law $\nu(\dd z)\1_{\{|z|>n\}}/\nu(\{|z|>n\})$.
For $p>1$, conditioning on $N_n$ and using $(\sum_{j=1}^m a_j)^p\le m^{p-1}\sum_{j=1}^m a_j^p$ together with the independence of $N_n$ and $(Z_{n,j})_j$ and the standard Poisson moment bound
$\E[N_n^p]\le C_p(\lambda_n+\lambda_n^p)$, we obtain
\[
\E\Big[\Big(\int_0^T\int_{|z|>n}|z|\,N(\dd s,\dd z)\Big)^p\Big]
\le C_{p,T}\int_{|z|>n}|z|^p\,\nu(\dd z).
\]
Therefore,
\[
\E\Big[\sup_{0\le t\le T}|\Delta^{(n)}_t|^p\Big]\le C\,\|\kappa\|_\infty^p\int_{|z|>n}|z|^p\,\nu(\dd z)\xrightarrow[n\to\infty]{}0
\]
by Assumption~\ref{ass:Levy}\,(L2).

\medskip
\noindent\emph{Step 5: passage to the limit and identification of $D_\Theta X$.}
Let $Y$ be the unique adapted c\`adl\`ag solution of
\[
Y_t=\int_0^t \partial_x b(s,X_s)Y_s\,\dd s + C_t,\qquad Y_0=0.
\]
(Existence and uniqueness follow from a pathwise Gronwall argument;
moreover $\sup_{t\le T}|Y_t|\le e^{KT}\sup_{t\le T}|C_t|$, which belongs to $L^p$ by Step~3.)
Set $E^{(n)}_t:=D_\Theta X^{(n)}_t-Y_t$.
From \eqref{eq:DXjump-n} we obtain
\[
E^{(n)}_t=\int_0^t \partial_x b(s,X^{(n)}_s)E^{(n)}_s\,\dd s +\int_0^t\big(\partial_x b(s,X^{(n)}_s)-\partial_x b(s,X_s)\big)Y_s\,\dd s,
\]
and hence
\[
\sup_{0\le t\le T}|E^{(n)}_t|
\le e^{KT}\int_0^T|\partial_x b(s,X^{(n)}_s)-\partial_x b(s,X_s)|\,|Y_s|\,\dd s.
\]
By Step~4, $\E[\sup_{t\le T}|X^{(n)}_t-X_t|^p]\to 0$, so we can extract a subsequence $n_k$ such that
$\sup_{t\le T}|X^{(n_k)}_t-X_t|\to0$ a.s.
Since for each $s$ the map $x\mapsto \partial_x b(s,x)$ is continuous (and $\partial_x b$ is jointly measurable by the Carath\'eodory property) and $\partial_x b$ is bounded, the integrand converges to $0$ a.s.\ and is dominated by $2K\sup_{t\le T}|Y_t|\in L^p$.
Therefore, by dominated convergence,
\[
\E\Big[\sup_{0\le t\le T}|D_\Theta X^{(n_k)}_t-Y_t|^p\Big]\to0.
\]
Finally, for each fixed $t$ Step~2 gives $X^{(n_k)}_t\in W^{1,p}_\Theta$, so the pairs $(X^{(n_k)}_t, D_\Theta X^{(n_k)}_t)$ belong to the graph of $D_\Theta$.
Since $X^{(n_k)}_t\to X_t$ in $L^p$ and $D_\Theta X^{(n_k)}_t\to Y_t$ in $L^p$, these pairs converge in $L^p\times L^p$.
By the closability of $D_\Theta$ on $L^p$ (equivalently, closedness of the graph of its closure; see \cite[Section~3]{NakagawaSuzuki2024}),
we conclude that $X_t\in W^{1,p}_\Theta$ and $D_\Theta X_t=Y_t$.
Since the limit $Y_t$ is determined by $X$ alone (as the unique solution of the linear equation above), it is independent of the chosen subsequence.
Thus $D_\Theta X_t=Y_t$ for each $t\in[0,T]$; in particular, \eqref{eq:DXjump-linear} holds and $D_\Theta X$ admits a c\`adl\`ag version.
\end{proof}

\begin{lemma}[Expression for the directional derivative]\label{lem:DX-jump-expression}
Under Assumptions~\ref{ass:Levy} and~\ref{ass:jump}, with $\Theta=(0,v)$ as above,
the directional derivative satisfies
\begin{equation}\label{eq:DXjump}
D_\Theta X_t
= \exp\Big(\int_0^t \partial_x b(s,X_s)\,\dd s\Big)\,
\int_0^t\int_{\R_0}
\kappa(s)^2\exp\Big(-2\int_0^s \partial_x b(u,X_u)\,\dd u\Big)\eta(z)\,N(\dd s,\dd z).
\end{equation}
In particular, $D_\Theta X_t\ge 0$ for all $t\in[0,T]$.
\end{lemma}

\begin{proof}
Fix $p\in(1,\beta)$.
By Lemma~\ref{lem:jump-differentiability}, $X_t\in W^{1,p}_\Theta$ for all $t\in[0,T]$ and $D_\Theta X$ satisfies the linear equation \eqref{eq:DXjump-linear}.
Set $J_t^b:=\exp(\int_0^t \partial_x b(s,X_s)\,\dd s)$.
Since $\partial_x b$ is bounded, $J^b$ is a strictly positive process of bounded variation.
Applying the product rule to $(J_t^b)^{-1}D_\Theta X_t$ and using \eqref{eq:DXjump-linear} gives $\dd\big((J_t^b)^{-1}D_\Theta X_t\big)=(J_t^b)^{-1}\,\dd C_t$,
where $C_t:=\int_0^t\int_{\R_0}\kappa(s)\,v(s,z)\,N(\dd s,\dd z)$.
Integrating and multiplying by $J_t^b$ yields
\[
D_\Theta X_t=\exp\Big(\int_0^t \partial_x b(s,X_s)\,\dd s\Big)\int_0^t\!\!\int_{\R_0}
\kappa(s)\exp\Big(-\int_0^s \partial_x b(u,X_u)\,\dd u\Big)v(s,z)\,N(\dd s,\dd z).
\]
Substituting our choice $v(s,z)=\kappa(s)\exp(-\int_0^s \partial_x b(u,X_u)\,\dd u)\eta(z)$ gives \eqref{eq:DXjump}.
Since $\eta\ge 0$ and $N$ is a nonnegative measure, the Poisson integral in \eqref{eq:DXjump} is nonnegative, and consequently $D_\Theta X_t\ge 0$.
\end{proof}

\subsection{A positivity lemma with time-dependent weight}

The next lemma is the main probabilistic ingredient of the pure-jump regime.
Even if $\kappa$ vanishes on subintervals, condition \eqref{eq:NDtime} ensures that each interval $(0,t]$ contains a subset of positive Lebesgue measure on which $\kappa\neq 0$.
Combined with infinite activity of $\nu$ on any neighborhood of zero and the positivity of $\eta$, this forces the weighted Poisson integral to accumulate strictly positive mass almost surely.

\begin{lemma}[Weighted Poisson integral is strictly positive]\label{lem:poisson-positive}
Assume that Assumption~\ref{ass:Levy} holds.
Let $\kappa:[0,T]\to\R$ be a bounded \emph{deterministic} measurable function satisfying \eqref{eq:NDtime}.
Let $\eta\ge 0$ be measurable such that $\nu(\{z:\eta(z)>0\})=\infty$.
Then for every $t\in(0,T]$,
\[
\int_0^t\int_{\R_0} \kappa(s)^2\,\eta(z)\,N(\dd s,\dd z) > 0
\qquad\text{a.s.}
\]
where the integral is understood as an extended nonnegative random variable (possibly $+\infty$).
\end{lemma}

\begin{proof}
Fix $t\in(0,T]$ and set
\[
B_t:=\{s\in(0,t]:\kappa(s)\neq 0\}.
\]
Since $\kappa$ is a deterministic measurable function, $B_t$ is a deterministic Lebesgue measurable subset of $(0,t]$
(this suffices for the Poisson parameter computation below, since the intensity measure of $N$ is $\mathrm{Leb}\otimes\nu$).
Note that $|B_t|>0$ by \eqref{eq:NDtime}: indeed, if $|B_t|=0$ then $\kappa=0$ a.e.\ on $(0,t]$, hence $\int_0^t \kappa(s)^2\,\dd s=0$, contradicting \eqref{eq:NDtime}.
Let $A:=\{z\in\R_0:\eta(z)>0\}$, so that $\nu(A)=\infty$ by assumption.

Since $\nu(\{|z|>1\})<\infty$ by Assumption~\ref{ass:Levy}\,(L1) and $\nu(A)=\infty$, we have
$\nu(A_\infty)=\infty$ for $A_\infty:=A\cap\{0<|z|\le 1\}$.
By $\sigma$-finiteness of $\nu$, choose an increasing sequence of Borel sets $A_n\uparrow A_\infty$
such that $\nu(A_n)<\infty$ for each $n$ and $\nu(A_n)\uparrow\infty$.
For instance, one may take $A_n:=A\cap\{1/n\le |z|\le 1\}$.

For each $n$, the random variable $N(B_t\times A_n)$ is Poisson with parameter $|B_t|\nu(A_n)$.
Therefore
\[
\Pbb\big(N(B_t\times A_n)=0\big)=\exp\big(-|B_t|\nu(A_n)\big)\xrightarrow[n\to\infty]{}0.
\]
Since $A_n\uparrow A_\infty$, monotone convergence of measures gives $N(B_t\times A_n)\uparrow N(B_t\times A_\infty)$,
and hence
\[
\Pbb\big(N(B_t\times A_\infty)=0\big)=\lim_{n\to\infty}\Pbb\big(N(B_t\times A_n)=0\big)=0.
\]
Since $A_\infty\subset A$, the event $\{N(B_t\times A)\ge 1\}\supset\{N(B_t\times A_\infty)\ge 1\}$ has probability one.
On this event there exists at least one jump point $(s,z)$ with $s\in B_t$ and $z\in A$.
For such a point, $\kappa(s)^2\eta(z)>0$, and hence the Poisson integral is strictly positive:
\[
\int_0^t\int_{\R_0} \kappa(s)^2\,\eta(z)\,N(\dd s,\dd z)>0.
\]
This proves the claim.
\end{proof}

\subsection{Completion of the proof}

We now combine the explicit representation of $D_\Theta X$, the positivity lemma, and the moment/Sobolev estimates to check the assumptions of Proposition~\ref{prop:unified}.

\begin{lemma}[Moment bounds needed for the criterion]\label{lem:jump-moment}
Fix $p\in(1,\beta)$ and assume that Assumptions~\ref{ass:Levy} and~\ref{ass:jump} hold.
Then $\Xstar_T\in L^p(\Omega)$.
Moreover, for $\Theta=(0,v)$ defined above, $D_\Theta X$ admits a right-continuous version and
\[
\E\Big[\sup_{0\le t\le T}|D_\Theta X_t|^p\Big]<\infty.
\]
\end{lemma}

\begin{proof}
\emph{Moment bound for $\Xstar_T$.}
By the Lipschitz property of $b$ and Gronwall's lemma, there exists $C<\infty$ such that
\[
\sup_{0\le t\le T}|X_t|\le C\Big(1+|x|+\sup_{0\le u\le T}\Big|\int_0^u \kappa(s)\,\dd L_s\Big|\Big).
\]
Since $\kappa$ is bounded and $p\in(1,\beta)$, the L\'evy--It\^o decomposition and standard BDG estimates for Poisson integrals
(as in the proof of Lemma~\ref{lem:apriori}, with $\sigma\equiv0$) yield
$\E[\sup_{t\le T}|\int_0^t \kappa(s)\,\dd L_s|^p]<\infty$.
Therefore $\Xstar_T\in L^p(\Omega)$.

\medskip
\emph{Moment bound for $D_\Theta X$.}
By \eqref{eq:DXjump}, $D_\Theta X_t$ is the product of a pathwise continuous exponential and a Poisson integral
$\int_0^t\int f(s,z)\,N(\dd s,\dd z)$, which is c\`adl\`ag in $t$ (being a sum over jump points).
Hence $D_\Theta X$ admits a right-continuous (in fact c\`adl\`ag) version.
From \eqref{eq:DXjump} and $\|\partial_x b\|_\infty=:K<\infty$ we have, for all $t\le T$,
\[
0\le D_\Theta X_t
\le e^{KT}\int_0^t\int_{\R_0} \kappa(s)^2 e^{2Ks}\eta(z)\,N(\dd s,\dd z)
\le C\int_0^T\int_{\R_0}\eta(z)\,N(\dd s,\dd z),
\]
with $C=C(K,T,\|\kappa\|_\infty)$.
Since $\eta$ is bounded, compactly supported in $\{|z|\le 1\}$ and satisfies \eqref{eq:eta-choice} (in particular, $\eta(z)\lesssim |z|^4$ near $0$),
Assumption~\ref{ass:Levy}\,(L1) implies $\int \eta(z)\,\nu(\dd z)<\infty$ and $\int \eta(z)^p\,\nu(\dd z)<\infty$.
To bound its $p$th moment, write
\[
\int_0^T\int \eta(z)\,N(\dd s,\dd z)
= T\int \eta(z)\,\nu(\dd z)+\int_0^T\int \eta(z)\,\tilde N(\dd s,\dd z).
\]
Hence, by the triangle inequality and $(|a|+|b|)^p\le C_p(|a|^p+|b|^p)$,
\[
\E\Big[\Big(\int_0^T\int \eta(z)\,N(\dd s,\dd z)\Big)^p\Big]
\le C_p\Big(T^p\Big|\int \eta(z)\,\nu(\dd z)\Big|^p+\E\Big[\Big|\int_0^T\int \eta(z)\,\tilde N(\dd s,\dd z)\Big|^p\Big]\Big).
\]
The compensated integral
\[
M_T:=\int_0^T\int \eta(z)\,\tilde N(\dd s,\dd z)
\]
is a purely discontinuous martingale.
If $p\in(1,2]$, then Lyapunov's inequality and the It\^o isometry yield
\[
\E[|M_T|^p]\le \E[|M_T|^2]^{p/2}
=\Big(T\int \eta(z)^2\,\nu(\dd z)\Big)^{p/2}<\infty.
\]
If $p>2$, then the Burkholder--Davis--Gundy inequality together with Kunita's first inequality gives
\[
\E[|M_T|^p]\le C_p\Big( \Big(T\int \eta(z)^2\,\nu(\dd z)\Big)^{p/2}
+T\int \eta(z)^p\,\nu(\dd z)\Big)<\infty.
\]
This is finite since $\eta$ is bounded, supported in $\{|z|\le 1\}$ and $\eta(z)\lesssim |z|^4$ near~$0$.
Therefore,
\[
\E\Big[\Big(\int_0^T\int \eta(z)\,N(\dd s,\dd z)\Big)^p\Big]<\infty,
\]
and the claim follows.
\end{proof}

\begin{proof}[Proof of Theorem~\ref{thm:jump}]
Fix $p\in(1,\beta)$ and take the direction $\Theta=(0,v)$ defined above.
By Lemma~\ref{lem:jump-differentiability} we have $X_t\in W^{1,p}_\Theta$ for all $t\in[0,T]$ and $D_\Theta X$ admits a c\`adl\`ag version;
Lemma~\ref{lem:DX-jump-expression} yields the explicit representation \eqref{eq:DXjump}.
Lemma~\ref{lem:jump-moment} gives $\Xstar_T\in L^p(\Omega)$ and $\E[\sup_{t\le T}|D_\Theta X_t|^p]<\infty$.

\medskip
It remains to prove strict positivity of $D_\Theta X_t$ for all $t\in(0,T]$.
By monotonicity of the Poisson integral in time, it suffices to verify positivity on the countable dense subset $\mathbb{Q}_T:=\mathbb{Q}\cap(0,T]$.
Define
\[
\Omega_0:=\bigcap_{q\in\mathbb{Q}_T}\left\{\int_0^q\int \kappa(s)^2 \exp\Big(-2\int_0^s \partial_x b(u,X_u)\,\dd u\Big)\eta(z)\,N(\dd s,\dd z)>0\right\}.
\]
For each fixed $q\in\mathbb{Q}_T$, set $K:=\|\partial_x b\|_\infty$.
Since
\[
\exp\Big(-2\int_0^s \partial_x b(u,X_u)\,\dd u\Big)\ge \exp(-2Ks)\ge \exp(-2KT),\qquad s\le q,
\]
we have
\[
\int_0^q\int \kappa(s)^2 \exp\Big(-2\int_0^s \partial_x b(u,X_u)\,\dd u\Big)\eta(z)\,N(\dd s,\dd z)
\ge \exp(-2KT)\int_0^q\int \kappa(s)^2\eta(z)\,N(\dd s,\dd z).
\]
By Lemma~\ref{lem:poisson-positive}, the right-hand side is strictly positive a.s., and therefore so is the left-hand side.
Therefore each event in the intersection has probability one and $\Pbb(\Omega_0)=1$.

On $\Omega_0$, for any $t\in(0,T]$ choose $q\in\mathbb{Q}_T$ with $0<q\le t$.
Since the integrand is nonnegative, we obtain
\begin{align*}
    &\int_0^t\int \kappa(s)^2 \exp\Big(-2\int_0^s \partial_x b(u,X_u)\,\dd u\Big)\eta(z)\,N(\dd s,\dd z)\\
&\ge \int_0^q\int \kappa(s)^2 \exp\Big(-2\int_0^s \partial_x b(u,X_u)\,\dd u\Big)\eta(z)\,N(\dd s,\dd z)>0.
\end{align*}

Since the exponential prefactor in \eqref{eq:DXjump} is strictly positive, this implies $D_\Theta X_t>0$ for all $t\in(0,T]$ a.s.
Proposition~\ref{prop:unified} now applies and yields absolute continuity of $\law(\Xstar_T)$.
\end{proof}

\section{Proof of Theorem~\ref{thm:silent}}\label{sec:proof-silent}

We observe that if both noise coefficients vanish on an initial interval, then the dynamics is deterministic there.
The running maximum therefore starts from a deterministic level $m_0$, and the post-$\delta$ maximum contributes an atom at $m_0$.

\begin{proof}[Proof of Theorem~\ref{thm:silent}]
Fix $t\in[0,\delta]$.
Since $\sigma(t,x)=0$ for a.e.\ $t\in[0,\delta]$ and all $x$, the Brownian integral vanishes:
$\E\big[|\int_0^t \sigma(s,X_s)\,\dd W_s|^2\big]=\E[\int_0^t \sigma(s,X_s)^2\,\dd s]=0$, so $\int_0^t \sigma(s,X_s)\,\dd W_s=0$ a.s.
Likewise, by the It\^o isometry for compensated Poisson integrals,
\[
\E\Big[\Big|\int_0^t\!\!\int_{|z|\le 1}\kappa(s)z\,\tilde N(\dd s,\dd z)\Big|^2\Big]
=\int_0^t \kappa(s)^2\,\dd s\int_{|z|\le 1} z^2\,\nu(\dd z)=0,
\]
so $\int_0^t\int_{|z|\le 1}\kappa(s)z\,\tilde N(\dd s,\dd z)=0$ a.s.
For the large-jump part, set $B:=\{s\in[0,\delta]:\kappa(s)\neq 0\}$; then $|B|=0$, and $N(B\times\{|z|>1\})$ is Poisson with mean $|B|\nu(\{|z|>1\})=0$, so $N(B\times\{|z|>1\})=0$ a.s.
Since $\kappa(s)=0$ for $s\notin B$ and $N(B\times\{|z|>1\})=0$ a.s., we conclude that
$\int_0^t\int_{|z|>1}\kappa(s)z\,N(\dd s,\dd z)=0$ a.s.
Consequently, on $[0,\delta]$ the SDE reduces to the ODE $\dd X_t=b(t,X_t)\dt$.
By uniqueness of solutions to this ODE (guaranteed by the Lipschitz assumption on $b$), $X$ is deterministic on $[0,\delta]$.
In particular $\Xstar_\delta=m_0$ is deterministic.
The displayed identities follow.
If the post-$\delta$ dynamics satisfies Theorem~\ref{thm:elliptic} or Theorem~\ref{thm:jump},
then $\sup_{\delta\le s\le T}X_s$ has an absolutely continuous law, so the distribution of $\Xstar_T=\max(m_0,\sup_{\delta\le s\le T}X_s)$
is the mixture of an atom at $m_0$ and an a.c.\ component on $(m_0,\infty)$.
\end{proof}

\appendix
\section{Technical lemmas}\label{sec:appendix}

For completeness, we collect the estimates used repeatedly in the main proofs:
a priori moment bounds for solutions and Jacobians, admissibility of the directions, the weighted-Poisson positivity lemma,
and the differentiability statement for the truncated pure-jump approximation.

\subsection{Moment bounds for the solution and the Jacobian}

\begin{lemma}[A priori estimates]\label{lem:apriori}
Under Assumption~\ref{ass:elliptic} and for any $p\in(1,\beta)$, there exists $C_{p,T}<\infty$ such that
\[
\E\Big[\sup_{0\le t\le T}|X_t|^p\Big]\le C_{p,T}(1+|x|^p).
\]
\end{lemma}

\begin{proof}
Fix $p\in(1,\beta)$ and set $K_b:=\|\partial_x b\|_\infty$, $B_0:=\sup_{t\in[0,T]}|b(t,0)|$.
By (E1) we have the linear growth bound
\begin{equation}\label{eq:lin-growth-b}
|b(t,x)|\le B_0+K_b|x|,\qquad (t,x)\in[0,T]\times\R.
\end{equation}
Write the SDE in integral form and take the running supremum:
\[
\sup_{0\le t\le T}|X_t|
\le |x|+\int_0^T |b(s,X_s)|\,\dd s+\sup_{0\le t\le T}\Big|\int_0^t \sigma(s,X_s)\,\dd W_s\Big|
+\sup_{0\le t\le T}\Big|\int_0^t \kappa(s)\,\dd L_s\Big|.
\]
Using $(a_1+\cdots+a_m)^p\le C_p\sum_{i=1}^m a_i^p$ and \eqref{eq:lin-growth-b}, we obtain
\begin{align*}
\E\Big[\sup_{0\le t\le T}|X_t|^p\Big]
&\le C_{p}\Big(
|x|^p + \E\Big[\Big(\int_0^T (B_0+K_b|X_s|)\,\dd s\Big)^p\Big]
+ \E\Big[\sup_{0\le t\le T}\Big|\int_0^t \sigma(s,X_s)\,\dd W_s\Big|^p\Big]\\
&\hspace{4.2em}
+ \E\Big[\sup_{0\le t\le T}\Big|\int_0^t \kappa(s)\,\dd L_s\Big|^p\Big]\Big).
\end{align*}

\medskip
\noindent For the drift term, by H\"older's inequality and \eqref{eq:lin-growth-b},
\[
\begin{aligned}
\Big(\int_0^T (B_0+K_b|X_s|)\,\dd s\Big)^p
&\le 2^{p-1}(B_0T)^p + 2^{p-1}K_b^pT^{p-1}\int_0^T |X_s|^p\,\dd s\\
&\le C(1+|x|^p)+C\int_0^T \sup_{0\le u\le s}|X_u|^p\,\dd s.
\end{aligned}
\]
Taking expectations and using $\E[\sup_{u\le s}|X_u|^p]\le \E[\sup_{u\le T}|X_u|^p]$, we obtain
\begin{equation}\label{eq:drift-bound}
\E\Big[\Big(\int_0^T |b(s,X_s)|\,\dd s\Big)^p\Big]
\le C(1+|x|^p)+C\int_0^T \E\Big[\sup_{0\le u\le s}|X_u|^p\Big]\,\dd s.
\end{equation}

\medskip
\noindent For the Brownian term, by the BDG inequality and boundedness of $\sigma$,
\begin{equation}\label{eq:bdg-brown}
\E\Big[\sup_{0\le t\le T}\Big|\int_0^t \sigma(s,X_s)\,\dd W_s\Big|^p\Big]
\le C_p\,\E\Big[\Big(\int_0^T \sigma(s,X_s)^2\,\dd s\Big)^{p/2}\Big]
\le C_p\,(\|\sigma\|_\infty^2T)^{p/2}.
\end{equation}

\medskip
\noindent For the L\'evy term, since $L$ has L\'evy triplet $(0,0,\nu)$, we may write (L\'evy--It\^o decomposition)
\[
L_t=\int_0^t\int_{|z|\le 1} z\,\tilde N(\dd s,\dd z)+\int_0^t\int_{|z|>1} z\,N(\dd s,\dd z),
\]
and hence
\[
\int_0^t \kappa(s)\,\dd L_s
=\underbrace{\int_0^t\int_{|z|\le 1} \kappa(s)z\,\tilde N(\dd s,\dd z)}_{=:M_t}
+\underbrace{\int_0^t\int_{|z|>1} \kappa(s)z\,N(\dd s,\dd z)}_{=:I_t^{(>1)}}.
\]

\smallskip
\noindent For the small-jump part, the predictable quadratic variation of $M$ is
\[
\langle M\rangle_T=\int_0^T\int_{|z|\le 1} |\kappa(s)z|^2\,\nu(\dd z)\,\dd s
\le \|\kappa\|_\infty^2T\int_{|z|\le 1} z^2\,\nu(\dd z)<\infty
\]
by Assumption~\ref{ass:Levy}\,(L1).
By the BDG inequality with exponent $2$,
\[
\E\Big[\sup_{0\le t\le T}|M_t|^2\Big]\le C\,\E[\langle M\rangle_T]<\infty.
\]
If $p\in(1,2]$, Lyapunov's inequality yields
\begin{equation}\label{eq:smalljump-moment-p-le-2}
\E\Big[\sup_{0\le t\le T}|M_t|^p\Big]\le \E\Big[\sup_{0\le t\le T}|M_t|^2\Big]^{p/2}<\infty.
\end{equation}
If $p>2$, then the BDG inequality together with Kunita's first inequality for Poisson integrals (see, e.g., \cite[Theorem~4.4.23]{Applebaum2009}) gives
\begin{equation}\label{eq:bdg-smalljump}
\begin{aligned}
\E\Big[\sup_{0\le t\le T}|M_t|^p\Big]
&\le C_p\bigg[\Big(\int_0^T\int_{|z|\le 1} |\kappa(s)z|^2\,\nu(\dd z)\,\dd s\Big)^{p/2}
+\int_0^T\int_{|z|\le 1} |\kappa(s)z|^p\,\nu(\dd z)\,\dd s\bigg]\\
&\le C_p\bigg[\Big(\|\kappa\|_\infty^2T\int_{|z|\le 1} z^2\,\nu(\dd z)\Big)^{p/2}
+\|\kappa\|_\infty^p T\int_{|z|\le 1} |z|^p\,\nu(\dd z)\bigg]
<\infty,
\end{aligned}
\end{equation}
where the finiteness of the last term follows from Assumption~\ref{ass:Levy}\,(L1), since
$|z|^p\le |z|^2$ for $|z|\le 1$ when $p>2$.
Combining \eqref{eq:smalljump-moment-p-le-2} and \eqref{eq:bdg-smalljump} (depending on $p$) yields
\begin{equation}\label{eq:smalljump-moment}
\E\Big[\sup_{0\le t\le T}|M_t|^p\Big]<\infty.
\end{equation}

\smallskip
\noindent For the large-jump part, since $\nu(\{|z|>1\})<\infty$ (by $\int (1\wedge |z|^2)\,\nu(\dd z)<\infty$), the process
$I^{(>1)}$ is a compound Poisson process on $[0,T]$.
Using $\|\kappa\|_\infty<\infty$ and $\int_{|z|>1}|z|^p\,\nu(\dd z)<\infty$ (which follows from Assumption~\ref{ass:Levy}\,(L2)),
standard moment estimates for compound Poisson sums yield
\begin{equation}\label{eq:bigjump-moment}
\E\Big[\sup_{0\le t\le T}|I_t^{(>1)}|^p\Big]<\infty.
\end{equation}

\smallskip
\noindent Combining \eqref{eq:smalljump-moment} and \eqref{eq:bigjump-moment}, we obtain
\begin{equation}\label{eq:levy-bound}
\E\Big[\sup_{0\le t\le T}\Big|\int_0^t \kappa(s)\,\dd L_s\Big|^p\Big]<\infty.
\end{equation}

\medskip
Putting \eqref{eq:drift-bound}, \eqref{eq:bdg-brown} and \eqref{eq:levy-bound} together, we find
\[
\E\Big[\sup_{0\le t\le T}|X_t|^p\Big]
\le C(1+|x|^p)+C\int_0^T \E\Big[\sup_{0\le u\le s}|X_u|^p\Big]\,\dd s,
\]
for some constant $C=C(p,T,\|b(\cdot,0)\|_\infty,K_b,\|\sigma\|_\infty,\|\kappa\|_\infty,\nu)<\infty$.
Gronwall's lemma yields the desired bound.
\end{proof}

\subsection{Moment bounds without boundedness of the diffusion coefficient}

The boundedness assumption (E3) is convenient but not essential for the density argument, since our choice of
direction in \eqref{eq:hchoice} satisfies $|h(t)|=J_t^{-1}$ and does not involve $|\sigma(t,X_t)|$.
For completeness, we record a standard moment estimate under linear growth of $\sigma$.

\begin{lemma}[A priori estimates under linear growth]\label{lem:apriori-linear}
Assume that Assumption~\ref{ass:Levy} holds and that (E1), (E2) and (E4) of Assumption~\ref{ass:elliptic} hold
(without assuming (E3)).
Fix $p\in(1,\beta)$.
Then there exists a constant $C_{p,T}<\infty$ such that
\[
\E\Big[\sup_{0\le t\le T}|X_t|^p\Big]\le C_{p,T}(1+|x|^p).
\]
\end{lemma}

\begin{proof}
By (E1) there exist constants $B_0,K_b,\Sigma_0,K_\sigma<\infty$ such that
\[
|b(t,x)|\le B_0+K_b|x|,\qquad |\sigma(t,x)|\le \Sigma_0+K_\sigma|x|,\qquad (t,x)\in[0,T]\times\R.
\]
Let $f(t):=\E[\sup_{0\le s\le t}|X_s|^p]$.
Arguing as in the proof of Lemma~\ref{lem:apriori} (with the bounded coefficient $\|\sigma\|_\infty$ replaced by the linear-growth bound),
we obtain constants $C_1,C_2,C_3<\infty$ such that for all $t\in[0,T]$,
\begin{equation}\label{eq:f-ineq}
f(t)\le C_1(1+|x|^p)+C_2\int_0^t f(s)\,\dd s
+ C_3\,\E\Big[\Big(\int_0^t \sigma(s,X_s)^2\,\dd s\Big)^{p/2}\Big].
\end{equation}
By BDG and the growth bound on $\sigma$,
\[
\int_0^t \sigma(s,X_s)^2\,\dd s \le C\int_0^t (1+|X_s|^2)\,\dd s
\le C t\big(1+\sup_{0\le s\le t}|X_s|^2\big),
\]
and hence for some constant $C=C(p,\Sigma_0,K_\sigma)$,
\[
\E\Big[\Big(\int_0^t \sigma(s,X_s)^2\,\dd s\Big)^{p/2}\Big]
\le C\,t^{p/2}\,\E\Big[\big(1+\sup_{0\le s\le t}|X_s|^2\big)^{p/2}\Big]
\le C\,t^{p/2}\big(1+f(t)\big),
\]
where we used $(1+u^2)^{p/2}\le C_p(1+u^p)$ for $u\ge0$.
Plugging this into \eqref{eq:f-ineq} yields
\[
f(t)\le C(1+|x|^p)+C\int_0^t f(s)\,\dd s + C\,t^{p/2}f(t),\qquad t\in[0,T].
\]
Choose $\delta\in(0,T]$ such that $C\,\delta^{p/2}\le 1/2$.
Then for $t\in[0,\delta]$ we can absorb the last term and obtain
$f(t)\le C(1+|x|^p)+C\int_0^t f(s)\,\dd s$, and Gronwall's lemma implies $\sup_{t\in[0,\delta]}f(t)<\infty$.
To extend to $[0,T]$, partition $[0,T]$ into finitely many subintervals of length at most $\delta$
and apply the same estimate on each subinterval conditionally on the left endpoint (standard iteration argument).
This yields $f(T)\le C_{p,T}(1+|x|^p)$.
\end{proof}

\subsection{Derivation of the directional derivative in the elliptic regime}

\begin{lemma}[Directional derivative via an adapted shift]\label{lem:DXrep}
Assume that Assumption~\ref{ass:elliptic} holds and let $h\in \Hb_p$ for some $p>1$.
For $\varepsilon\in\R$ define $W_t^\varepsilon:=W_t+\varepsilon\int_0^t h(s)\,\dd s$ and let $X^\varepsilon$ be the unique strong
solution to
\begin{equation}\label{eq:SDE-eps}
X_t^\varepsilon=x+\int_0^t b(s,X_s^\varepsilon)\,\dd s+\int_0^t \sigma(s,X_s^\varepsilon)\,\dd W_s
+\varepsilon\int_0^t \sigma(s,X_s^\varepsilon)h(s)\,\dd s+\int_0^t \kappa(s)\,\dd L_s,\qquad t\in[0,T].
\end{equation}
Then, as $\varepsilon\to0$,
\[
\frac{X^\varepsilon-X}{\varepsilon}\longrightarrow Y
\quad\text{in }L^p\big(\Omega;C([0,T])\big),
\]
where $Y$ is the unique adapted solution to the linear SDE \eqref{eq:linearized}.
Moreover,
\[
Y_t=J_t\int_0^t J_s^{-1}\sigma(s,X_s)h(s)\,\dd s,\qquad t\in[0,T],
\]
where $J$ is the Jacobian defined by \eqref{eq:J}.
\end{lemma}

\begin{proof}
Let $\Lambda:=\|\partial_x b\|_\infty\vee \|\partial_x\sigma\|_\infty$.
Since $\partial_x b$ and $\partial_x\sigma$ are bounded, $b(t,\cdot)$ and $\sigma(t,\cdot)$ are globally Lipschitz with constant $\Lambda$.

\medskip
\noindent\emph{Step 1: first-order estimate for $X^\varepsilon-X$.}
Set $\Delta_t^\varepsilon:=X_t^\varepsilon-X_t$. Subtracting \eqref{eq:SDE-general} from \eqref{eq:SDE-eps} gives
\[
\Delta_t^\varepsilon=\int_0^t\!\big(b(s,X_s^\varepsilon)-b(s,X_s)\big)\,\dd s
+\int_0^t\!\big(\sigma(s,X_s^\varepsilon)-\sigma(s,X_s)\big)\,\dd W_s
+\varepsilon\int_0^t \sigma(s,X_s^\varepsilon)h(s)\,\dd s.
\]
Note that the jump terms cancel in the difference (the same L\'evy process $L$ drives both equations), so $\Delta^\varepsilon$ has \emph{continuous}
paths.
Using BDG, Lipschitz continuity, and the bound $|\sigma|\le\|\sigma\|_\infty$, we obtain for $p>1$ a constant $C=C(p,T,\Lambda,\|\sigma\|_\infty)$ such that
\begin{equation}\label{eq:Delta-eps-bound}
\E\Big[\sup_{0\le t\le T}|\Delta_t^\varepsilon|^p\Big]
\le C\int_0^T \E\Big[\sup_{0\le u\le s}|\Delta_u^\varepsilon|^p\Big]\,\dd s
+C|\varepsilon|^p\,\E\Big[\Big(\int_0^T |h(s)|^2\,\dd s\Big)^{p/2}\Big].
\end{equation}
By Gronwall's lemma,
\begin{equation}
\E\Big[\sup_{0\le t\le T}|\Delta_t^\varepsilon|^p\Big]\le C|\varepsilon|^p\|h\|_{\Hb_p}^p,
\end{equation}
so in particular $\Delta^\varepsilon\to0$ in $L^p(\Omega;C([0,T]))$ as $\varepsilon\to0$.

\medskip
\noindent\emph{Step 2: equation for the difference quotient.}
For $\varepsilon\neq 0$ define $Y_t^\varepsilon:=\Delta_t^\varepsilon/\varepsilon$. Define the predictable coefficients
\[
a_s^\varepsilon:=\int_0^1 \partial_x b\big(s, X_s+\theta\Delta_s^\varepsilon\big)\,\dd\theta,
\qquad
c_s^\varepsilon:=\int_0^1 \partial_x \sigma\big(s, X_s+\theta\Delta_s^\varepsilon\big)\,\dd\theta.
\]
Then
\[
\frac{b(s,X_s^\varepsilon)-b(s,X_s)}{\varepsilon}=a_s^\varepsilon\,Y_s^\varepsilon,
\qquad
\frac{\sigma(s,X_s^\varepsilon)-\sigma(s,X_s)}{\varepsilon}=c_s^\varepsilon\,Y_s^\varepsilon,
\]
and hence $Y^\varepsilon$ satisfies the linear SDE
\begin{equation}\label{eq:Yeps}
Y_t^\varepsilon=\int_0^t a_s^\varepsilon Y_s^\varepsilon\,\dd s+\int_0^t c_s^\varepsilon Y_s^\varepsilon\,\dd W_s+\int_0^t \sigma(s,X_s^\varepsilon)h(s)\,\dd s.
\end{equation}
In particular, $|a_s^\varepsilon|\vee |c_s^\varepsilon|\le \Lambda$.

\medskip
\noindent\emph{Step 3: convergence $Y^\varepsilon\to Y$.}
Let $Y$ denote the unique solution to \eqref{eq:linearized} and set $Z^\varepsilon:=Y^\varepsilon-Y$.
Subtracting \eqref{eq:linearized} from \eqref{eq:Yeps} yields
\[
Z_t^\varepsilon=\int_0^t a_s^\varepsilon Z_s^\varepsilon\,\dd s+\int_0^t c_s^\varepsilon Z_s^\varepsilon\,\dd W_s+R_t^\varepsilon,
\]
where
\[
R_t^\varepsilon:=\int_0^t (a_s^\varepsilon-a_s)Y_s\,\dd s+\int_0^t (c_s^\varepsilon-c_s)Y_s\,\dd W_s
+\int_0^t \big(\sigma(s,X_s^\varepsilon)-\sigma(s,X_s)\big)h(s)\,\dd s,
\]
and we write $a_s:=\partial_x b(s,X_s)$, $c_s:=\partial_x\sigma(s,X_s)$.
Using BDG and Gronwall as in \eqref{eq:Delta-eps-bound}, we obtain
\begin{equation}\label{eq:Zeps-gron}
\E\Big[\sup_{0\le t\le T}|Z_t^\varepsilon|^p\Big]\le C\,\E\Big[\sup_{0\le t\le T}|R_t^\varepsilon|^p\Big],
\end{equation}
with $C=C(p,T,\Lambda)<\infty$.

We now show $\E[\sup_{t\le T}|R_t^\varepsilon|^p]\to0$ as $\varepsilon\to0$ by treating the three terms in $R^\varepsilon$ in order.
For the first term, by H\"older's inequality,
\[
\sup_{0\le t\le T}\Big|\int_0^t (a_s^\varepsilon-a_s)Y_s\,\dd s\Big|
\le \int_0^T |a_s^\varepsilon-a_s|\,|Y_s|\,\dd s,
\]
so that
\begin{equation}\label{eq:R1}
\E\Big[\sup_{0\le t\le T}\Big|\int_0^t (a_s^\varepsilon-a_s)Y_s\,\dd s\Big|^p\Big]
\le T^{p-1}\int_0^T \E\big[|a_s^\varepsilon-a_s|^p\,|Y_s|^p\big]\,\dd s.
\end{equation}
Since $\Delta^\varepsilon\to0$ in probability uniformly on $[0,T]$, continuity and boundedness of $\partial_x b$ imply $a_s^\varepsilon\to a_s$ in probability for every $s$.
Moreover $|a_s^\varepsilon-a_s|\le 2\Lambda$.
We also note that $Y$ has finite $p$th moments in the supremum norm. Indeed, by BDG and $|a_s|\vee|c_s|\le \Lambda$,
\[
\E\Big[\sup_{0\le t\le T}|Y_t|^p\Big]
\le C\int_0^T \E\Big[\sup_{0\le u\le s}|Y_u|^p\Big]\,\dd s
+ C\,\E\Big[\Big(\int_0^T |\sigma(s,X_s)|^2|h(s)|^2\,\dd s\Big)^{p/2}\Big],
\]
for a constant $C=C(p,T,\Lambda)$.
Since $|\sigma|\le \|\sigma\|_\infty$ and $h\in \Hb_p$, the second term is bounded by $C\,\|\sigma\|_\infty^p\|h\|_{\Hb_p}^p<\infty$, and Gronwall's lemma yields $\E[\sup_{t\le T}|Y_t|^p]<\infty$.
Using $|Y_s|^p\le \sup_{0\le u\le T}|Y_u|^p$ and $\E[\sup_{u\le T}|Y_u|^p]<\infty$,
Vitali's theorem (uniform integrability) implies that the integrand in \eqref{eq:R1} converges to $0$ in $L^1$ for each $s$,
and is dominated by $(2\Lambda)^p\E[\sup_{u\le T}|Y_u|^p]$.
Hence the right-hand side of \eqref{eq:R1} tends to $0$ as $\varepsilon\to0$.

For the second term, the BDG inequality yields
\[
\E\Big[\sup_{0\le t\le T}\Big|\int_0^t (c_s^\varepsilon-c_s)Y_s\,\dd W_s\Big|^p\Big]
\le C_p\,\E\Big[\Big(\int_0^T |(c_s^\varepsilon-c_s)Y_s|^2\,\dd s\Big)^{p/2}\Big].
\]
Similarly, continuity and boundedness of $\partial_x\sigma$ yield $c_s^\varepsilon\to c_s$ in probability for every $s$, and $|c_s^\varepsilon-c_s|\le 2\Lambda$.
In particular, $c^\varepsilon-c\to0$ in measure on $[0,T]\times\Omega$ (with respect to $\dd s\otimes\Pbb$), hence along any sequence $\varepsilon_n\downarrow0$ one can extract a subsequence (not relabeled) such that $c_s^{\varepsilon_n}\to c_s$ for a.e.\ $(s,\omega)$.
Since for each fixed $\omega$ the path $s\mapsto Y_s(\omega)$ is continuous (hence bounded) on $[0,T]$, dominated convergence yields $\int_0^T |(c_s^{\varepsilon_n}-c_s)Y_s|^2\,\dd s\to0$ a.s.\ along the subsequence, and therefore the same integral converges to $0$ in probability as $\varepsilon\to0$.
Consequently $\big(\int_0^T |(c_s^\varepsilon-c_s)Y_s|^2\,\dd s\big)^{p/2}\to0$ in probability.
Moreover,
\[
\Big(\int_0^T |(c_s^\varepsilon-c_s)Y_s|^2\,\dd s\Big)^{p/2}
\le (4\Lambda^2)^{p/2}\Big(\int_0^T |Y_s|^2\,\dd s\Big)^{p/2}
\le (4\Lambda^2T)^{p/2}\sup_{0\le s\le T}|Y_s|^p,
\]
which is integrable. By Vitali's theorem (uniform integrability), we obtain
\[
\E\Big[\sup_{0\le t\le T}\Big|\int_0^t (c_s^\varepsilon-c_s)Y_s\,\dd W_s\Big|^p\Big]\longrightarrow 0.
\]

For the third term, the Cauchy--Schwarz inequality in $s$ gives
\[
\sup_{0\le t\le T}\Big|\int_0^t \big(\sigma(s,X_s^\varepsilon)-\sigma(s,X_s)\big)h(s)\,\dd s\Big|^p
\le \Big(\int_0^T |\sigma(s,X_s^\varepsilon)-\sigma(s,X_s)|^2\,\dd s\Big)^{p/2}\Big(\int_0^T |h(s)|^2\,\dd s\Big)^{p/2}.
\]
Under (E3), $|\sigma(s,X_s^\varepsilon)-\sigma(s,X_s)|\le 2\|\sigma\|_\infty$, hence
\[
\Big(\int_0^T |\sigma(s,X_s^\varepsilon)-\sigma(s,X_s)|^2\,\dd s\Big)^{p/2}\Big(\int_0^T |h(s)|^2\,\dd s\Big)^{p/2}
\le (4\|\sigma\|_\infty^2 T)^{p/2}\Big(\int_0^T |h(s)|^2\,\dd s\Big)^{p/2},
\]
where the right-hand side belongs to $L^1(\Omega)$ since $h\in \Hb_p$.
Moreover, since $\Delta^\varepsilon\to0$ in probability uniformly on $[0,T]$, Lipschitz continuity of $\sigma$
implies $\int_0^T |\sigma(s,X_s^\varepsilon)-\sigma(s,X_s)|^2\,\dd s\le \Lambda^2T\sup_{s\le T}|\Delta_s^\varepsilon|^2\to0$ in probability.
Therefore the product converges to $0$ in probability while being dominated by the $L^1(\Omega)$ random variable $(4\|\sigma\|_\infty^2 T)^{p/2}(\int_0^T |h(s)|^2\,\dd s)^{p/2}$.
By Vitali's theorem (uniform integrability implied by the $L^1$ domination), we obtain
\[
\E\Big[\sup_{0\le t\le T}\Big|\int_0^t \big(\sigma(s,X_s^\varepsilon)-\sigma(s,X_s)\big)h(s)\,\dd s\Big|^p\Big]\longrightarrow 0.
\]

Combining the three terms, we obtain $\E[\sup_{t\le T}|R_t^\varepsilon|^p]\to0$.
Hence, by \eqref{eq:Zeps-gron}, $\E[\sup_{t\le T}|Z_t^\varepsilon|^p]\to0$.
Thus $Y^\varepsilon\to Y$ in $L^p(\Omega;C([0,T]))$.

\medskip
\noindent\emph{Step 4: variation of constants.}
Let $J$ solve \eqref{eq:J}. Applying It\^o's formula to $J_t^{-1}Y_t$ and using \eqref{eq:J} and \eqref{eq:linearized} gives
$\dd(J_t^{-1}Y_t)=J_t^{-1}\sigma(t,X_t)h(t)\,\dd t$ and hence
\[
Y_t=J_t\int_0^t J_s^{-1}\sigma(s,X_s)h(s)\,\dd s.
\]
The proof is complete.
\end{proof}

\begin{lemma}[Identification of $D_\Theta X$ in the Brownian direction]\label{lem:DX-identification}
Assume that Assumption~\ref{ass:Levy} holds and that (E1), (E2) and (E4) of Assumption~\ref{ass:elliptic} hold (boundedness of $\sigma$ is not needed).
Fix $\Theta=(h,0)$ with $h\in \Hb_{\infty-}$ and let $X$ be the solution of \eqref{eq:SDE-general}.
Fix $p\in(1,\beta)$. Then for every $t\in[0,T]$ we have $X_t\in W^{1,p}_\Theta$ and the process $D_\Theta X$ is the unique adapted solution to \eqref{eq:linearized}.
In particular, $D_\Theta X$ admits a continuous version.
Moreover, if \textup{(E3)} additionally holds, the limit process $Y$ in Lemma~\ref{lem:DXrep} exists, and $D_\Theta X\equiv Y$.
\end{lemma}

\begin{proof}
We use the Wiener--Malliavin derivative with respect to the Brownian motion $W$ and compare it with the directional derivative $D_\Theta$
of \cite{NakagawaSuzuki2024,SongXie2018} in the purely Brownian direction $\Theta=(h,0)$.

\medskip
\noindent\emph{Step 1: the Brownian directional derivative on cylinder functionals.}
Let $C_p^\infty(\R^m)$ denote the class of $C^\infty$ functions whose derivatives have at most polynomial growth.
Following \cite[Section~2.2]{SongXie2018} (see also \cite[Definition~3.1]{NakagawaSuzuki2024}),
let $\mathcal{S}$ be the set of Wiener--Poisson cylinder functionals of the form
\[
F=f\big(W(\varphi_1),\dots,W(\varphi_{m_1}),\,N(g_1),\dots,N(g_{m_2})\big),
\]
where $f\in C_p^\infty(\R^{m_1+m_2})$, the functions $\varphi_i:[0,T]\to\R$ are bounded deterministic with $\varphi_i\in L^2(0,T)$,
and each $g_j:[0,T]\times\R_0\to\R$ is bounded deterministic, $C^1$ in $z$, with compact support in $z$ (so that $\supp g_j(\cdot,\cdot)\subset[0,T]\times K_j$ for some compact $K_j\subset\R_0$, hence $\nu(K_j)<\infty$; in particular $N(g_j)=\int_0^T\int_{\R_0}g_j(s,z)\,N(\dd s,\dd z)$ is well-defined even when $\nu(\R_0)=\infty$).
For $F\in\mathcal{S}$ we define the Wiener derivative $D^W F\in L^2(0,T)$ by
\[
D_s^W F=\sum_{i=1}^{m_1}\partial_i f(\cdots)\,\varphi_i(s),\qquad s\in[0,T].
\]
In the directional calculus of \cite{NakagawaSuzuki2024,SongXie2018}, for $\Theta=(h,0)$ the derivative of $F\in\mathcal{S}$ is
\[
D_\Theta F=\sum_{i=1}^{m_1}\partial_i f(\cdots)\int_0^T h(s)\varphi_i(s)\,\dd s
=\int_0^T D_s^W F\,h(s)\,\dd s.
\]
Note that $h$ may be random and adapted, so $D_\Theta F$ is itself a random variable even for $F\in\mathcal{S}$;
the identity holds pathwise for the Wiener component since the $\varphi_i$ are deterministic.

Fix $p\in(1,\beta)$ and choose $q$ with $p<q<\beta$.
Let $\mathbb{D}_W^{1,q}$ be the completion of $\mathcal{S}$ under the norm
$\|F\|_{L^q}+\|D^W F\|_{L^q(\Omega;L^2(0,T))}$.
For $F\in\mathbb{D}_W^{1,q}$ we set
\[
\mathcal{D}_\Theta F:=\int_0^T D_s^W F\,h(s)\,\dd s.
\]
By H\"older's inequality in $\omega$ and the Cauchy--Schwarz inequality in time,
\begin{equation}\label{eq:pairing-Lp}
\E[|\mathcal{D}_\Theta F|^p]
\le \E\Big[\Big(\int_0^T |D_s^W F|^2\,\dd s\Big)^{q/2}\Big]^{p/q}
\E\Big[\Big(\int_0^T |h(s)|^2\,\dd s\Big)^{\frac{pq}{2(q-p)}}\Big]^{(q-p)/q},
\end{equation}
so $\mathcal{D}_\Theta F\in L^p$ since $h\in \Hb_{\infty-}$.
Indeed, because $h\in \Hb_r$ for every $r\ge 1$, choosing $r=\frac{pq}{q-p}$ makes the second expectation in \eqref{eq:pairing-Lp} finite.
Moreover, if $F_n\to F$ in $\mathbb{D}_W^{1,q}$ with $F_n\in\mathcal{S}$, then \eqref{eq:pairing-Lp} applied to $F_n-F$ shows that
$\mathcal{D}_\Theta F_n\to\mathcal{D}_\Theta F$ in $L^p$.
Therefore $F\in W^{1,p}_\Theta$ and $D_\Theta F=\mathcal{D}_\Theta F$.

\medskip
\noindent\emph{Step 2: Wiener differentiability of the SDE solution.}
By classical Malliavin differentiability results for Brownian SDEs with $C^1$ coefficients with bounded spatial derivatives (see, e.g., \cite{IshikawaKunita2006,Nualart2006}) and noting that the additive jump term in \eqref{eq:SDE-general} is independent of $W$ and thus acts as an adapted inhomogeneity, the solution $X$ of \eqref{eq:SDE-general} belongs to $\mathbb{D}_W^{1,q}$ for every $q\in(1,\beta)$.
Its Wiener derivative $D_r^W X_t$ satisfies, for $0\le r\le t\le T$,
\begin{equation}
D_r^W X_t=\sigma(r,X_r)+\int_r^t \partial_x b(s,X_s)\,D_r^W X_s\,\dd s+\int_r^t \partial_x\sigma(s,X_s)\,D_r^W X_s\,\dd W_s.
\end{equation}
(When (E3) is not assumed, the initial condition $D_r^W X_r=\sigma(r,X_r)$ has $q$th moments controlled by Lemma~\ref{lem:apriori-linear} only for $q<\beta$, whence the restriction $q\in(1,\beta)$.)
Moreover, with the Jacobian $J$ from Lemma~\ref{lem:J} we have the explicit representation
\begin{equation}\label{eq:DXW-Jacobian}
D_r^W X_t=J_t\,J_r^{-1}\,\sigma(r,X_r)\,\1_{\{r\le t\}}.
\end{equation}
Indeed, applying It\^o's formula to $J_t^{-1}D_r^W X_t$ on $[r,T]$ and using $D_r^W X_r=\sigma(r,X_r)$ yields \eqref{eq:DXW-Jacobian}.

\medskip
\noindent\emph{Step 3: identification of $D_\Theta X$ and the linear SDE.}
Fix $p\in(1,\beta)$ and choose $q\in(p,\beta)$.
Since $X_t\in \mathbb{D}_W^{1,q}$ by Step~2, Step~1 yields $X_t\in W^{1,p}_\Theta$ and
\[
D_\Theta X_t=\int_0^T D_r^W X_t\,h(r)\,\dd r=\int_0^t D_r^W X_t\,h(r)\,\dd r.
\]
Using \eqref{eq:DXW-Jacobian} and writing $G_t:=\int_0^t J_r^{-1}\sigma(r,X_r)h(r)\,\dd r$, we obtain $D_\Theta X_t=J_t\,G_t$.
Since $\dd G_t=J_t^{-1}\sigma(t,X_t)h(t)\,\dd t$, It\^o's formula and the SDE for $J$ show that $D_\Theta X$ satisfies \eqref{eq:linearized}.
The uniqueness of the adapted solution to \eqref{eq:linearized} follows from standard BDG/Gronwall estimates, and continuity is immediate.

\medskip
\noindent\emph{Step 4: identification with the shift derivative (bounded $\sigma$).}
If (E3) holds, Lemma~\ref{lem:DXrep} yields an $L^p(\Omega;C([0,T]))$-limit $Y$ of the difference quotient $(X^\varepsilon-X)/\varepsilon$.
By Lemma~\ref{lem:DXrep}, this $Y$ is an adapted solution to \eqref{eq:linearized}.
By uniqueness, $D_\Theta X\equiv Y$.
\end{proof}
\section*{Statements and Declarations}
\textbf{Funding} This work was supported by JSPS KAKENHI Grant Number 23K12507.

\noindent\textbf{Competing interests} The authors have no competing interests to declare that are relevant to the content of this article.

\noindent\textbf{Data availability} Data sharing is not applicable to this article, as no datasets were generated or analyzed during the current study.

\bibliographystyle{spmpsci}
\bibliography{references}

\end{document}